\documentclass[12pt]{article}

\usepackage{hyperref}
\usepackage{ifthen}

\newcommand{\version}{arxiv}
\ifthenelse{\equal{\version}{general}}{
\usepackage[dvipsnames]{xcolor}
\usepackage[dvips]{graphicx}
\usepackage{media9}
\usepackage{float}
\usepackage{textcomp}
\usepackage{amsthm}
\usepackage{amsmath}
\usepackage{amssymb}
\usepackage{newtxmath}
\usepackage{bbm}
\usepackage[nameinlink,capitalise,noabbrev]{cleveref}

\usepackage{geometry}
\geometry{verbose,tmargin=3.5cm,bmargin=3.5cm,lmargin=3.5cm,rmargin=3.5cm}

\marginparwidth 0pt
\oddsidemargin  0pt
\evensidemargin  0pt
\marginparsep 0pt
\topmargin   -.5in
\hoffset -0.1in
\textwidth   6.6in
\textheight  8.8in

\newtheorem{theorem}{Theorem}[section]
\newtheorem{lemma}[theorem]{Lemma}
\newtheorem{assumption}{Assumption}
\newtheorem{corollary}[theorem]{Corollary}
\newtheorem{proposition}[theorem]{Proposition}
\newtheorem{remark}[theorem]{Remark}
\newtheorem{definition}[theorem]{Definition}
}

\ifthenelse{\equal{\version}{siam}}{
\input{Siam_files/macros}
}

\ifthenelse{\equal{\version}{arxiv}}{
\usepackage{media9}
\usepackage[dvipsnames]{xcolor}
\usepackage[pdftex]{graphicx}
\usepackage{float}
\usepackage{textcomp}
\usepackage{amsthm}
\usepackage{amsmath}
\usepackage{amssymb}
\usepackage{newtxmath}
\usepackage{bbm}
\usepackage[nameinlink,capitalise,noabbrev]{cleveref}

\marginparwidth 0pt
\oddsidemargin  0pt
\evensidemargin  0pt
\marginparsep 0pt
\topmargin   -.5in
\hoffset -0.1in
\textwidth   6.6in
\textheight  8.8in

\newtheorem{theorem}{Theorem}[section]

\newtheorem{assumption}{Assumption}
\newtheorem{corollary}[theorem]{Corollary}
\newtheorem{proposition}[theorem]{Proposition}
\newtheorem{remark}[theorem]{Remark}
\newtheorem{definition}[theorem]{Definition}
}

\usepackage[normalem]{ulem}
\usepackage{booktabs}
\usepackage{rotating}

\usepackage[backend=bibtex,maxnames=10]{biblatex}
\usepackage{multirow}
\usepackage{tabularx}
\usepackage[all]{xy}

\usepackage[commentColor=blue]{algpseudocodex}
\usepackage{algorithm}

\usepackage{scalefnt}
\usepackage[shortlabels]{enumitem}

\usepackage{xspace}
\newcommand{\solar}[1]{${\sf solar #1}$\xspace}

\newcommand{\N}{\mathbb{N}}

\newcommand{\Q}{\mathbb{Q}}
\newcommand{\R}{\mathbb{R}}
\newcommand{\Z}{\mathbb{Z}}

\newcommand{\K}{\mathcal{K}}
\newcommand{\M}{\mathcal{M}}

\newcommand{\B}{\mathcal{B}}
\newcommand{\Ll}{\mathcal{L}}

\newcommand{\Ss}{\mathcal{S}}
\newcommand{\Dd}{\mathcal{D}}
\newcommand{\T}{\mathcal{T}}

\newcommand{\bvec}[1]{{\bf{ #1}}}
\newcommand{\vecx}{{\bf{x}}}
\newcommand{\barx}{\bar{\vecx}}
\newcommand{\bard}{\bar{\vecd}}
\newcommand{\vecy}{{\bf{y}}}
\newcommand{\vecd}{{\bf{d}}}
\newcommand{\vecs}{{\bf{s}}}
\newcommand{\vecw}{{\bf{w}}}
\newcommand{\vecv}{{\bf{v}}}
\newcommand{\vecz}{{\bf{z}}}
\newcommand{\vecp}{{\bf{p}}}
\newcommand{\vecgreek}[1]{\text{\boldmath$#1$}}

\usepackage{mathtools}

\newcommand{\sumgext}{\displaystyle\sum_{\ell\in\Gext}}
\newcommand{\Omegalog}{\Omega^{\tt int}}
\newcommand{\Omegaext}{\Omega^{\tt ext}}
\newcommand{\Glog}{\mathcal{G}^{\tt int}}
\newcommand{\Gext}{\mathcal{G}^{\tt ext}}
\newcommand{\cext}{c^{\tt ext}}
\newcommand{\cint}{c^{\tt int}}

\newcommand{\norm}[1]{\|#1\|}

\newcommand{\set}[1]{\{#1\}}

\newcommand{\F}{\mathcal{F}}

\usepackage{soul} 
\usepackage{colortbl} 
\definecolor{Red}{rgb}{1,0,0}
\definecolor{Gray}{rgb}{0.2,0.2,0.2}
\definecolor{Maroon}{rgb}{0.6,0.05,0.03}
\definecolor{Blue}{rgb}{0,0.7,0.9}
\definecolor{Green}{rgb}{0,.7,0}

\definecolor{orange}{rgb}{1.00,0.41,0.16}
\newcommand{\blue}{\color{blue}}

\title{A penalty-interior point method combined with MADS for equality and inequality constrained optimization}

\author{Charles Audet\thanks{GERAD and Department of Mathematics and Industrial Engineering, Polytechnique Montr\'eal. 6079, Succ. Centre-ville Montr\'eal, Qu\'ebec H3C~3A7, Canada (\href{mailto:charles.audet@gerad.ca}{charles.audet@gerad.ca}, 
\href{mailto:youssef.diouane@polymtl.ca}{youssef.diouane@polymtl.ca},
\href{mailto:sebastien.le-digabel@polymtl.ca}{sebastien.le-digabel@polymtl.ca},
\href{mailto:christophe.tribes@polymtl.ca}{christophe.tribes@polymtl.ca})}%
\and
Andrea Brilli\thanks{``Sapienza'' University of Rome, Department of Computer Control and Management Engineering ``A. Ruberti'', Rome, Italy (\href{mailto:brilli@diag.uniroma1.it}{brilli@diag.uniroma1.it}). First and corresponding author.}
\and
Youssef Diouane\footnotemark[1]
\and
S\'ebastien {Le~Digabel}\footnotemark[1]
\and
Everton J. Silva \thanks{NOVA University of Lisbon, Center for Mathematics and Applications (Nova Math), Campus da Caparica, 2829-516, Caparica, Portugal (\href{mailto:ejo.silva@campus.fct.unl.pt}{ejo.silva@campus.fct.unl.pt}).\\
ORCID: 0000-0002-3043-5393 (Audet), 0000-0002-8307-4106 (Brilli), 0000-0002-6609-7330 (Diouane), 0000-0003-3148-5090 (Le~Digabel), 0000-0002-3627-3030 (Silva), 0000-0002-8740-6155 (Tribes).
}
\and
Christophe Tribes\footnotemark[1]
}

\date{\small \today}

\newcommand{\textabstract}{
This work introduces MADS-PIP, an efficient framework that integrates a penalty-interior point strategy into the mesh adaptive direct search (MADS) algorithm for solving nonsmooth blackbox optimization problems with general inequality and equality constraints.
Inequality constraints are partitioned into two subsets: one treated via a logarithmic barrier applied to an aggregated interior constraint violation, and the other handled through an exterior quadratic penalty.
All equality constraints are treated by the exterior penalty. 
A merit function defines a sequence of unconstrained subproblems, which are solved approximately using MADS, while a carefully designed update rule drives the penalty-barrier parameter to zero. 
In the nonsmooth setting, we establish convergence results ensuring feasibility for general constraints as well as Clarke stationarity for inequality-constrained problems. 
Computational experiments on both analytical test sets and challenging blackbox problems demonstrate that the proposed MADS-PIP algorithm is competitive with, and often outperforms, MADS with the progressive barrier strategy, particularly in the presence of equality constraints.
}
\begin{document}

\maketitle

\ifthenelse{\NOT{\equal{\version}{siam}}}{
\begin{abstract}
\textabstract\\


\end{abstract}}{
\begin{abstract}
\textabstract
\end{abstract}
}

\section{Introduction}

The presence of general constraints in optimization introduces challenges. 
For nonlinear or nonconvex constraints, classical optimization methods, such as gradient-based ones, may be inefficient~\cite{Be1982}. 
The task is even harder when derivatives are unavailable. 
Algorithms that do not rely on derivative information are particularly relevant for problems where the objective and/or the constraint functions are defined by complex simulations, for which derivatives are not computable, unreliable, or simply nonexistent.
Such nonsmooth problems are addressed by Derivative-Free Optimization (DFO) algorithms~\cite{AuHa2017,CoScVi2009}. 
This work targets constrained optimization problems of the form
\begin{align}
\displaystyle\min_{\vecx \in \mathbb{R}^n} \quad & f(\vecx) \nonumber\\
\text{s.t.} \quad & g_\ell(\vecx)\leq 0, & \ell=1,2,\dots,m \label{P0}\tag{$P_0$}\\
& h_j(\vecx)=0,  & j=1,2,\dots,p \nonumber
\end{align}
where $f:\R^{n}\to\overline\R$, $g_\ell:\R^n\to\overline\R$ for all $\ell=1,2,\ldots,m$, $h_j:\R^n\to\overline\R$ for all $j=1,2,\ldots,p$, and $\overline\R = \R\cup\{+\infty\}$. 

\subsection{Constraint handling in DFO}

Over the years, several methods have been proposed to handle constraints, most of them employing penalty, filter, or barrier approaches~\cite{NoWr2006}.

In the context of direct search methods, early contributions to smooth constrained problems were proposed by Lewis and Torczon for bound-constrained~\cite{LeTo1999} and linearly constrained~\cite{LeTo2000} problems.
They proposed an augmented Lagrangian approach~\cite{LeTo2002} for general constraints, using a sequence of bound-constrained minimizations.
Audet and Dennis~\cite{AuDe2004} propose a filter approach~\cite{FlLe2002} that accepts a trial point which improves either the objective function value or the constraint violation measure.

Mesh Adaptive Direct Search (MADS)~\cite{AuDe2006} is a generalization of pattern search methods that handles nonsmooth constraints using an extreme barrier approach that only evaluates feasible points. 
To address the limitations of the extreme barrier, a progressive barrier approach was proposed in~\cite{AuDe2009}.
Similar to a filter approach, the progressive barrier method sets a threshold on the constraint violation accepted, progressively decreasing its value as the iterations progress. The approach was later extended to handle linear equality constraints~\cite{AuLedPey2015}, by reformulating the optimization problem, potentially reducing the number of original variables.

Merit-function techniques have been investigated within the derivative-free literature;
for instance,~\cite{GrVi2014} proposes a merit-based strategy for DFO methods, while~\cite{DIOUANE2021100001} extends a class of evolution strategies to handle quantifiable relaxable constraints~\cite{LedWild2015} through a merit function combined with a restoration procedure, and addresses unrelaxable constraints using either an extreme-barrier or projection approach.

Linesearch approaches are directional direct search methods. In this framework, a contribution to smooth bound-constrained problems was proposed by Lucidi and Sciandrione~\cite{Lucidi2002}. Slight modifications of the same approach have been recently proposed in~\cite{BrilliCristfofari2025}, inspired by~\cite{Brilli2024}, proving worst-case complexity bounds and a finite active-set identification property.
In~\cite{LiLu2009} the authors consider inequality constrained problems using a transformation into unconstrained or linearly constrained minimization of a smooth approximation of a nonsmooth exact penalty function. In~\cite{FaLiLuRi2014} an exact penalty approach addressing nonsmooth inequalities is proposed.
A sequential penalty approach was proposed in~\cite{LiLuSc2010}.  It solves the original problem through a sequence of approximate minimizations of a merit function, progressively increasing the penalization of constraint violations.
{In~\cite{BrLiLu2025}, using the same algorithmic framework, a merit function handles inequality constraints with a logarithmic barrier approach and equality constraints by including a penalization term}. The latter approach has been recently integrated into direct search (LOG-DS) to address general nonlinear constrained optimization problems~\cite{BrCuLiSi2024}, assuming the continuous differentiability of the functions. The convergence to KKT-stationary points was established for LOG-DS under smooth assumptions. 

\subsection{Contributions and structure of the paper}

DFO algorithms, particularly direct search methods, are well known for their ability to handle nonsmooth problems. This work extends the MADS algorithm to equality constrained optimization problems.
The constraints are partitioned into two sets.
One set is treated by a logarithmic barrier, and the other is handled by an exterior penalty.
In the nonsmooth setting, we derive convergence guarantees that lead to feasibility under general constraints and to Clarke stationarity in the case of inequality constraints.
Computational results obtained on both analytical benchmarks and realistic blackbox problems show that the proposed method is competitive with, and often outperforms, the MADS progressive barrier approach~\cite{AuDe2009}, especially in the presence of equality constraints.

\Cref{sec-algo} describes the proposed algorithm, with emphasis on the partition of inequality constraints.
\Cref{sec-conv} develops the convergence analysis, addressing both feasibility for general constraints and stationarity for inequality constraints. 
\Cref{sec-Impl} describes practical implementation aspects, and \cref{sec-tests} reports computational experiments on analytical test problems and difficult blackbox problems. 
\Cref{sec-conc} concludes with final remarks.

\section{A penalty-interior point method for MADS}
\label{sec-algo}

This work proposes to solve~\eqref{P0} by means of a merit function approach 
that transforms the constrained problem into a sequence of unconstrained ones. 

\subsection{Partitioning the constraints}

The approach partitions the inequality constraints indices into two sets, namely $\Glog$ and $\Gext$, so that $\Glog\cap\Gext=\emptyset$ and $\Glog\cup\Gext=\set{1,2,\dots,m}$. 
This partition depends on the starting point, and further details are provided in \cref{sec-Impl}. 
The set of feasible points for~\eqref{P0} is the intersection of 
\begin{eqnarray*}
\Omegalog &=&\{\vecx\in\R^n\mid \ g_{\ell}(\vecx) \leq 0, \ \forall\ell\in\Glog \}\\
\mbox{and } \quad \Omegaext &=&\left\{ \vecx\in\R^n \mid \ 
g_{\ell}(\vecx) \leq 0, \  \forall\ell\in\Gext; \ 
h_j(\vecx) = 0, \  \forall j=1,2,\dots,p
\right\}.
\end{eqnarray*}
Therefore, a point $\vecx\in\R^n$ is feasible for~\eqref{P0} if and only if $\vecx\in\Omega\equiv\Omegalog\cap\Omegaext$. 

We briefly recall the logarithmic-barrier/exterior-penalty strategy used in~\cite{BrCuLiSi2024}.
The approach analyzed in~\cite{BrLiLu2025} employs a merit function that treats the inequality constraints via a logarithmic barrier and the equality constraints via an exterior penalty term. A well-known drawback of a pure barrier strategy is the requirement of a strictly feasible starting point; this is addressed in~\cite{BrCuLiSi2024} by combining the logarithmic term with an exterior penalty that handles the constraints violated by the initial point. Another practical issue is that the use of the logarithmic term can potentially cause the merit function to be unbounded below. Since the method is based on minimizing such a function, this is a clear issue. 

This paper exploits the nonsmooth settings to address this difficulty by imposing a threshold on the constraint values, namely by considering $\min\{t, \, -g_\ell(\vecx)\}$, where $t>0$ is a fixed scalar. Although the choice of $t$ lies outside the scope of this work, we note that setting $t=1$ ensures that the logarithm of $\min\{t, \, -g_\ell(\vecx)\}$ is nonnegative for any $\vecx\in\R^n$ that strictly satisfies the constraint $g_\ell(\vecx) < 0$. The same effect of changing $t$ is achieved by scaling the constraint functions. 
Therefore, without loss of generality, we fix $t=1$ for the remainder of the paper. With this choice, the constraints indexed by $\Glog$ are aggregated into the {\em interior constraint violation} function
$$\cint (\vecx)\ =\begin{cases}
-\prod_{\ell\in\Glog}\ \min\left\{1, \, -g_\ell(\vecx)\right\} & \text{if } g_\ell(\vecx)\leq0 \text{ for all } \ell\in\Glog,\\
\phi^{\tt prox}(\vecx) & \text{otherwise},   
\end{cases}
$$
where $\phi^{\tt prox}(\vecx) =\max_{\ell\in\Glog} g_\ell(\vecx)$. The function $\phi^{\tt prox}(\cdot)$ represents a measure of proximity to the boundary of $\Omegalog$, assuming negative values for all points $\vecx$ that strictly satisfy the inequality constraints in $\Glog$, and positive values if at least one of such constraints is violated. The interior constraint violation function satisfies $\cint(\vecx)\in[-1,+\infty[$ for all $\vecx\in\R^n$, and $\cint(\vecx)\leq0$ if and only if $\vecx\in\Omegalog$. Therefore, under this transformation, the set $\Omegalog$ can be equivalently defined as $\Omegalog=\{\vecx\in\R^n\mid \cint(\vecx)\leq 0\}$. Furthermore, if for all $\ell\in\Glog$, the functions $g_\ell(\vecx)$ are (Lipschitz) continuous near a point $\vecx\in\R^n$, then both $\phi^{\tt prox}(\cdot)$ and $\cint(\cdot)$ are also (Lipschitz) continuous around $\vecx$. 

The proposed method deals with the constraint $\cint(\vecx)\leq0$ using the transformation $\log(-\cint(\vecx))$, thus treating it with a strictly feasible approach. Conceptually, this strategy is similar to the formulations in~\cite{BrCuLiSi2024,BrLiLu2025}, where all inequality constraints assigned to the logarithmic barrier appear individually in the merit function through the sum $\sum_{\ell} \log\left(-g_\ell(\vecx)\right)$. 
In fact, the properties of the logarithm ensure that, for all $\vecx\in\Omegalog$
{\small
$$\log(-\cint(\vecx)) = \log\left(\prod_{\ell\in\Glog}\ \min\left\{1, \, -g_\ell(\vecx)\right\}\right)= \sum_{\ell\in\Glog} \log\left(\min\{1,-g_\ell(\vecx)\}\right)\in[0,+\infty[,$$}
which differs only in the threshold described in the previous paragraph.

The remaining constraints are addressed by penalizing their violation, using the {\em exterior constraint violation} function
\begin{equation}\label{eq:violation_function}
\cext(\vecx) \ =\  
\sumgext \left(\max\{0,g_{\ell}(\vecx)\}\right)^2+\sum_{j=1}^{p} h_j(\vecx)^2\ \geq \ 0
\end{equation}
that satisfies $\cext(\vecx) = 0$ if and only if $\vecx \in \Omegaext$.
Given these two constraint violation functions, we introduce a merit function that solves~\eqref{P0} through a sequence of unconstrained problems.

\begin{definition}
\label{merit-z}
The merit function $z: \R^n\times \R_+ \to \overline\R$ associated to problem~\eqref{P0} is defined as
\begin{equation*}
\displaystyle z(\vecx;\rho) \ =\ 
\begin{cases}
f(\vecx) - \rho\log(-\cint(\vecx)) +
\frac{1}{\rho} \cext(\vecx) \qquad ~& \text{if } \cint(\vecx)<0,\\
+\infty \quad & \text{ otherwise},
\end{cases}
\end{equation*} 
where $\cint$ and $\cext$ are the interior and exterior constraint violation functions.
\end{definition}

Note that the penalty parameter $\rho$ is the same for both terms of $z(\cdot;\rho)$. 
The key property of this merit function is that
for any $\vecx \in \R^n$, 
$$ \lim_{\rho \searrow 0} \, z(\vecx; \rho) \ = \ 
\left\{ \begin{array}{cl}
f(\vecx) & \mbox{ if } \vecx \in \Omegalog_< \cap \Omegaext ,\\
+\infty & \mbox{ otherwise,}
\end{array}\right.$$
where $\Omegalog_< = \{ \vecx \in \R^n \mid \ \cint(\vecx)<0 \}$. 
The merit function is close, but not identical to the extreme barrier function~\cite{AuDe2006}, as the latter is equal to $f$ on the entire feasible set, including when $\vecx$ satisfies $\cint(\vecx)=0$.

\subsection{The algorithm: MADS-PIP}

The Penalty-Interior Point (PIP) method consists of generating a sequence of penalty-barrier parameters $\set{\rho_k}$, and finding a sequence of inexact solutions $\set{\vecx_k}\in\R^n$ to subproblems of the form
\begin{align}
\displaystyle \min_{\vecx\in\R^n}\ z(\vecx;\rho_k).
\label{Pz}
\tag{$P_{\rho_k}$}\\
\nonumber
\end{align}
\vspace*{-10mm}

The paper provides conditions ensuring that as $k$ goes to $\infty$, the sequence of solutions $\set{\vecx_k}$ converges to a point satisfying nonsmooth necessary optimality conditions. 
In order for that to happen, the merit function $z(\cdot;\rho_k)$ must converge to the objective function $f$.
This occurs when $\rho_k$ converges to~$0$.
\Cref{LOG-MADS-1} uses MADS to approximately solve the unconstrained subproblems, 
with an additional update criterion for updating the penalty-barrier parameter $\rho_k$. 
MADS stops solving subproblem ($P_{\rho_k}$) when $\Delta_{k+1}$ is sufficiently small.
When the penalty-barrier parameter is reduced (i.e., $\rho_{k+1}< \rho_k$), the algorithm proceeds to the next {subproblem~($P_{\rho_{k+1}}$).}
The strategy is described in the next paragraphs.


\begin{algorithm}
\caption{{\bf Penalty-Interior Point with MADS (MADS-PIP)} \label{LOG-MADS-1}}
\begin{algorithmic}[0]

\Require Starting point/incumbent $\vecx_0\in\R^n$, partition sets $\Glog$ and $\Gext$ such that $g_\ell(\vecx_0)<0$ for all $\ell\in\Glog$, penalty contraction parameter $\theta_{\rho}\in(0,1)$, initial frame size $\Delta_0>0$, frame size contraction/expansion parameter $\theta_{\Delta}\in(0,1)\cap\Q$, $\beta>1$, set $k=0$.
\LComment{{\bf Step~0}: Mesh definition}
\State Set $\M_k = \{\vecx_k + \delta_k \vecz  \, : \, \vecz \in \Z^n\}$ where $\delta_k = \min\left\{\Delta_k,\frac{\Delta_k^2}{\Delta_0}\right\}$
\LComment{{\bf Step~1}: Search step (optional)}
\If{$z(\bvec{s};\rho_k) < z(\vecx_k;\rho_k)$ holds for some $\bvec{s}$ in a finite set $\Ss_k\subset\mathcal{M}_k$}
\State $\vecx_{k+1} \gets \bvec{s}$,\  $\Delta_{k+1} \gets \theta_{\Delta}^{-1}\Delta_{k}$, \ $\rho_{k+1} \gets \rho_k$
\State declare the iteration as {\bf successful} and go to {\bf \blue Step~4}
\EndIf
\LComment{{\bf Step~2}: Poll step}
\State Select a finite set of poll directions $\Dd_k$ such that $\set{\vecx_k+\vecd_k\mid \vecd_k\in\Dd_k}\subseteq\F_k$
\State with $\F_k=\set{\vecx\in\M_k\mid \norm{\vecx_k-\vecx}\leq \Delta_k}$
\If{$z(\vecx_k+\vecd_k;\rho_k) < z(\vecx_k;\rho_k)$ holds for some $\vecd_k\in\Dd_k$}
\State $\vecx_{k+1} \gets \vecx_k+\vecd_k$, \ $\Delta_{k+1} \gets \theta_{\Delta}^{-1}\Delta_{k}$, \ $\rho_{k+1} \gets \rho_k$
\State declare the iteration as {\bf successful} and go to {\bf \blue Step~4}
\Else
\State $\vecx_{k+1} \gets \vecx_{k}$, \ $\Delta_{k+1} \gets \theta_{\Delta}\Delta_{k}$
\State declare the iteration as {\bf unsuccessful}  
\EndIf
\LComment{{\bf Step~3}: Penalty parameter update (at unsuccessful iterations)}
\If{$\Delta_{k+1} \le \min\left\{\rho_k^{\beta}, [\phi^{\tt prox}(\vecx_k)]^2\right\}$}
\State $\rho_{k+1} \gets \theta_{\rho}\rho_k$ \Comment{$\vecx_{k+1} = \vecx_k$ is an approximate solution to subproblem ($P_{\rho_k}$)}  
\Else
\State $\rho_{k+1} \gets \rho_k$
\Comment{The subproblem ($P_{\rho_k}$) is not satisfactorily solved\hspace{17.8mm}}  
\EndIf
\LComment{{\bf Step~4}: Stopping criterion check}
\If{some stopping criterion is satisfied} \State{\textbf{Stop}} \Else 
\State $k \gets k+1 $ and  go to {\bf \blue Step~0}
\EndIf
\end{algorithmic}
\end{algorithm}

\Cref{LOG-MADS-1} builds upon the MADS algorithm~\cite{AuDe2006,AuHa2017}. 
Thanks to the use of the mesh, MADS accepts any strict decrease in the objective function value to declare an iteration successful. 
More recently, the ADS algorithm~\cite{G-2025-53} proposes a strategy that does not require a mesh nor any sufficient decrease.

Starting from an initial guess $\vecx_0$ and an initial \textit{frame size} $\Delta_0$, \cref{LOG-MADS-1} generates a sequence of iterates $\set{\vecx_k}$.
The best known solution at the start of iteration $k$ is denoted $\vecx_k$ and is called the \textit{incumbent}.
At each iteration $k$, the frame size is used to define the mesh $\M_k$ ({\bf \blue Step~0}), as with OrthoMADS~\cite{AbAuDeLe09}. 
Then, the optional \textit{search step} ({\bf \blue Step~1}) is performed, wherein various heuristics, e.g., quadratic models~\cite{CoLed2011}, surrogates~\cite{AuLedSa21}, Latin Hypercube sampling~\cite{McCoBe79a} and Nelder-Mead~\cite{AuTr2018}, are exploited in order to improve the solution, i.e., the search attempts to find a point $\vecs\in\M_k$ satisfying $z(\bvec{s};\rho_k) < z(\vecx_k;\rho_k)$, which would become the incumbent  of the next iteration. 
If the search step succeeds, the algorithm proceeds directly to the iteration $k+1$, repeating the search step. 
If it fails, {\bf \blue Step~2}, i.e. the \textit{poll step}, is invoked.

The poll step is restricted to the frame $\mathcal{F}_k=\set{\vecx\in\M_k\mid \norm{\vecx_k-\vecx}\leq \Delta_k}$ around the incumbent $\vecx_k$. 
A set of poll directions $\Dd_k$ is generated, so that the set of trial points $\{\vecx_k+\vecd_k\mid \vecd_k\in\Dd_k\}$ is a subset of $\F_k$, and thus $\norm{\vecd_k}\leq\Delta_k$ for all $\vecd_k\in\Dd_k$. If there exists a poll direction $\vecd_k\in\Dd_k$ such that $\vecx_k+\vecd_k$ improves the value of the function $z(\cdot;\rho_k)$, then the iteration is declared successful and the algorithm proceeds to iteration $k+1$, otherwise the iteration is declared unsuccessful and the frame size is reduced by a factor $\theta_\Delta\in(0,1)$. Whenever an iteration is unsuccessful, $\F_k$ is called a \textit{minimal frame}, and the incumbent point $\vecx_k$ is called a \textit{minimal frame center}. The novelty of the mixed penalty-barrier approach comes into play when iterations are unsuccessful.
In such case, {\bf \blue Step~3} is invoked. 

In {\bf \blue Step~3}, the algorithm verifies whether this measure is sufficiently small by balancing two quantities. 
The parameter $\rho_k$ represents how well the problem~\eqref{Pz} represents~\eqref{P0}. 
The closer $\rho_k$ is to zero, the better the representation. Therefore, we are not interested in providing a high precision for the stationarity measure if $\rho_k$ is large.
Furthermore, it is important to avoid trial points {outside of $\Omegalog$, where} the function $z(\cdot;\rho_k)$ is infinite. To achieve this, the method takes into account the value of $\phi^{\tt prox}(\vecx)$, which is a measure of proximity to the boundary of $\Omegalog$.
The criterion to determine if the frame size parameter is sufficiently small is
\begin{equation}\label{eq:penalty_criterion}
\Delta_{k+1} \ \leq \ \min\set{\rho_k^\beta,\, [\phi^{\tt prox}(\vecx_k)]^2}.
\end{equation}
If the latter is satisfied, the barrier-penalty parameter is decreased by a factor $\theta_\rho\in(0,1)$, otherwise it remains the same and the algorithm proceeds to the next iteration on the same subproblem ($P_{\rho_k}$). The criterion~\eqref{eq:penalty_criterion} ensures that if the parameter $\rho_k$ tends to zero, so does the frame size $\Delta_k$, at least on the subsequence where the criterion is satisfied. A trivial consequence of the updating rule of $\rho_k$ is that, if the criterion is satisfied for an infinite subset of iterations, then $\rho_k \to 0$.
Furthermore, we prove in \cref{prop:feasibility_limit} that~\eqref{eq:penalty_criterion} also ensures that, for $k$ large enough, the entire minimal frame belongs to the interior of $\Omegalog$, i.e., $\cint(\vecs)<0$ for all $\vecs\in\F_k$.

We conclude the algorithm description with brief comments on the initialization of the sets of indices $\Glog$ and $\Gext$. It is necessary that $g_\ell(\vecx_0)<0$ for all $\ell\in\Glog$, otherwise one could have $\cint(\vecx_0)\geq 0$ which may result in $z(\vecx_0;\rho_0)=+\infty$. This requirement can be satisfied by simply setting $\ell\in\Glog$ if $g_\ell(\vecx_0)<0$, and $\ell\in\Gext$ otherwise.
This strategy produces $\Glog=\emptyset$ when $g_\ell(\vecx_0)\geq0$ for all $\ell=1, 2, \dots, m$.

Finally, observe that the partition sets $\Glog$ and $\Gext$ may undergo a finite number of updates during the execution of the algorithm, as detailed in \cref{subsection::switch}.
The convergence analysis is then carried out on the sequence of iterations following the last change to the partition.
To simplify the notation and without loss of generality, 
we omitted the iteration subscript $k$ in the definition of the sets $\Glog$ and $\Gext$ in \cref{LOG-MADS-1}.

\section{Convergence analysis}
\label{sec-conv}

This section is devoted to the convergence properties of the sequence of iterates $\set{\vecx_k}$ produced by \cref{LOG-MADS-1}. The first part introduces the framework of Clarke's directional derivatives, which is a suitable concept for nonsmooth functions. Additionally, the \textit{path-refining} directions are defined. For such directions, it is possible to derive properties that can be exploited to understand the behaviour at the limit of the sequences generated by the Algorithm. Then, a subsection studies conditions under which the feasibility of end-path points can be ensured. 
Finally, the last subsection provides further conditions ensuring stationarity results.

\subsection{Preliminary results}

This section presents the key results required for the convergence analysis of \cref{LOG-MADS-1}.
We begin with a standard assumption ensuring that MADS behaves properly on the associated unconstrained subproblem.

\begin{assumption}
\label{ass:bounded_level_sets}
The lower level sets $\mathcal{L}_f(\alpha) = \{\vecx\in\R^n\mid f(\vecx) \leq \alpha \}$ are bounded for all $\alpha\in\R$.
\end{assumption}

\cref{ass:bounded_level_sets} implies that the lower level sets $\mathcal{L}_z(\alpha;\rho) = \{\vecx\in\R^n\mid z(\vecx;\rho) \leq \alpha \}$ are bounded for all $\rho>0$. In fact, by definition of $z(\cdot;\rho)$, one has $z(\vecx;\rho)\geq f(\vecx)$ for all $\vecx\in\R^n$ and all positive values of~$\rho$, which implies that $\Ll_z(\alpha;\rho)\subseteq \Ll_f(\alpha)$ for all $\alpha\in\R$ and $\rho\in(0,+\infty)$.

The convergence analysis in this work does not establish properties for the entire sequence of iterates generated by the algorithm. As in the case of Generalized Pattern Search (GPS) and MADS~\cite{AuHa2017}, where specific \textit{refining} subsequences are considered, we define an index set corresponding to the iterations at which the penalty parameter $\rho_k$ is updated -- these are the unsuccessful iterations for which convergence results are derived. 
They are named \textit{path-following} subsequences, a term adopted from the literature on interior-point methods (see, e.g.,~\cite{Bertsekas_1999}), to describe sequences of approximate solutions to the barrier subproblems. In addition, to guide the reader, we introduce a second index set that identifies the accumulation points associated with the sequence of incumbents.

\begin{definition}
The sequence $\set{\vecx_k}_{k\in \K_\rho}$ of iterates generated by~\cref{LOG-MADS-1}, corresponding to iterations where $\rho_{k+1}<\rho_k$, is said to be a path-following subsequence. Furthermore, a point $\barx$ is said to be an end-path point if there exists
$\K_\rho^{\tt x}\subseteq \K_\rho$ such that $\lim_{k\in\K_\rho^{\tt x}}\vecx_k = \bar{\vecx}$. In such case, $\set{\vecx_k}_{k\in\K_\rho^{\tt x}}$ is referred to as an end-path subsequence.
\end{definition}

{\bf \blue Step~3} of \cref{LOG-MADS-1} ensures that all iterations $k\in\K_\rho$ of a path-following subsequence are unsuccessful. Thus, since $\K_\rho^{\tt x}\subseteq \K_\rho$, we have for all $k\in\K_\rho^{\tt x}$
\begin{eqnarray}\label{eq:uns_zeta}
&z(\vecx_k+\vecd_k;\rho_k) \geq  z(\vecx_k;\rho_k)& {\mbox{for all}}~ \vecd_k\in\Dd_k,
\\
&\Delta_{k+1} \leq \min\set{\rho_k^\beta,[\phi^{\tt prox}(\vecx_{k})]^2} \quad
\text{ and }& \quad \Delta_{k+1} = \theta_\Delta\Delta_k. \nonumber
\end{eqnarray}
In the following, we refer to the set of successful iteration indexes as $\K_s$, and to the set of unsuccessful indexes as $\K_u$. Note that $\K_{\rho}^{\tt x}\subseteq\K_{\rho}\subseteq\K_u$. The first result to be shown is that the index set $\K_\rho$ is infinite, which we prove to be equivalent to showing that the penalty parameter $\rho_k$ goes to zero as $k\to\infty$. Therefore, we prove that the updating rule, i.e. 
$\Delta_{k+1} \le\min\{
\rho_k^{\beta},[\phi^{\tt prox}(\vecx_k)]^2\}$
with $k\in\K_u$, in {\bf \blue Step~3} of \cref{LOG-MADS-1}, is eventually satisfied for any value of $\rho$. Within the proof, we exploit the known results for MADS algorithms. In particular, let us state the following proposition. The proof is omitted as the result follows from~\cite[Proposition 3.4]{AuDe03a}.

\begin{proposition}
\label{prop:liminf_delta_mads}
Let \cref{ass:bounded_level_sets} be satisfied. 
If the sequence of penalty-barrier parameters $\set{\rho_k}$ produced by \cref{LOG-MADS-1} satisfies $\rho_k = \rho_0$ for all  $k \in \N$, then
$\liminf_{k\to\infty} \Delta_k = 0.$
\end{proposition}
Note that, when the penalty-barrier parameter is fixed to $\rho_0$ for all iterates, the algorithm is solving the unconstrained problem of minimizing $z(\vecx;\rho_0)$, and behaves as MADS. That is why it is not necessary to provide the proof of the latter proposition.

Let us now prove the following result on the convergence of the sequence of penalty-barrier parameters~$\set{\rho_k}$.

\begin{theorem}
\label{theo:convergence_tau}  
Let \cref{ass:bounded_level_sets} be satisfied. 
If $\set{\vecx_k}_{k\in\K_{\rho}}$ is a path-following subsequence generated by \cref{LOG-MADS-1}, then $\K_{\rho}$ is infinite and
\begin{equation}\label{eq:tau_to_0}
\displaystyle \lim_{k\to \infty} \rho_k=0.
\end{equation}
\end{theorem}
\begin{proof}
By contradiction, assume that $\K_{\rho}$ is finite. 
Without loss of generality, assume that $\K_{\rho}=\emptyset$. 
Therefore, $\rho_k=\rho_0$ for all $k \in \N$, and \cref{prop:liminf_delta_mads} ensures that $\liminf_{k\to\infty} \Delta_k = 0$.

Since $\Delta_{k+1} < \Delta_{k}$ if and only if $k\in\K_u$, it follows that
there exists an infinite index set $\K_u^\Delta\subseteq\K_u$ such that
\begin{equation}\label{eq:delta_to_zero_tau_0}
\lim_{k\in\K_u^\Delta} \Delta_k = 0.
\end{equation}
The condition of {\bf \blue Step~3} is never satisfied, and therefore 
{\small $\Delta_{k+1}>\min\{{\rho}_0^{\beta},[\phi^{\tt prox}(\vecx_k)]^2\}>0$}
for all $k\in\K_u^\Delta$.
Taking the limit for $k\to\infty$ and $k\in\K_u^\Delta$ and using~\eqref{eq:delta_to_zero_tau_0}, it follows that
\begin{equation*}
\lim_{k\in\K_u^\Delta} \phi^{\tt prox}(\vecx_k) = 0,
\quad \mbox{ and } \quad
\lim_{k\in\K_u^\Delta} \cint(\vecx_k) = 0,
\quad \mbox{ and } \quad
\lim_{k\in\K_u^\Delta} z(\vecx_k,\rho_0)=+\infty.
\end{equation*}
This contradicts the fact that $z(\vecx_{k+1};\rho_0)\leq z(\vecx_k;\rho_0)\leq z(\vecx_0;\rho_0) < +\infty$. 
Thus, the index set $\K_\rho$ is infinite.

For any $s \in \N$, let $k_s \in \K_\rho$ be the $s$-th iteration index in the set $\K_\rho$. 
Therefore, it follows that 
$\rho_{k_s} = \theta_\rho \rho_{k_{s-1}} = \theta_{\rho}^s \rho_0$, and $\lim_{s\to\infty} \rho_{k_s} = \lim_{s\to\infty} \theta_{\rho}^s \rho_0 = 0,$
since $\theta_\rho \in (0,1)$.
The result follows by noticing that the sequence $\{\rho_k\}$ is non-increasing.
\end{proof}
The next result shows that the limit of frame-size parameters is equal to zero for path-following subsequences.
\begin{theorem}
\label{theo:convergence_delta} 
Let \cref{ass:bounded_level_sets} be satisfied. 
If $\set{\vecx_k}_{k\in\K_{\rho}}$ is a path-following subsequence generated by \cref{LOG-MADS-1}, and $\{\Delta_k\}_{k\in\K_{\rho}}$
is the sequence of frame sizes, then
\begin{equation}\label{eq:Delta_to_0}
\displaystyle \lim_{k \in \K_\rho}  \Delta_k=0.
\end{equation}
\end{theorem}
\begin{proof}
\Cref{theo:convergence_tau} ensures that $\K_\rho$ is an infinite index set.
For every $k\in\K_\rho$, 
\begin{equation}\label{utile1}
\Delta_{k+1} \leq \min\{\rho_k^{\beta},[\phi^{\tt prox}(\vecx_k)]^{2}\}\leq \rho_k^{\beta},
\end{equation}
the same theorem guarantees that~\eqref{eq:tau_to_0} holds. 
Taking the limit as $k\to\infty$ and $k\in\K_\rho$ on~\eqref{utile1} gives
$\lim_{k\in\K_\rho}\Delta_{k+1} = 0$.
The algorithm sets $\Delta_{k+1} = \theta_{\Delta}\Delta_{k}$ for all $k \in \K_\rho\subseteq\K_u$, thus $\lim_{ k\in\K_\rho}\Delta_{k} = 0$.

\end{proof}

As a consequence of \cref{theo:convergence_delta}, we deduce that end-path points exist.
\begin{corollary}\label{cor:end_path_exist}
Let \cref{ass:bounded_level_sets} be satisfied. 
If $\set{\vecx_k}_{k\in\K_{\rho}}$ is a path-following subsequence generated by \cref{LOG-MADS-1}, then there exists at least one end-path point.
\end{corollary}

The results of~\cref{theo:convergence_tau},~\cref{theo:convergence_delta}, and~\cref{cor:end_path_exist} are of paramount importance and they are used throughout the analysis to derive further properties. In the following, Equation~\eqref{eq:Delta_to_0} is exploited to show additional preliminary results on the path-following sequence generated by \cref{LOG-MADS-1}. In particular, the minimal frames are proved to belong to the interior of $\Omegalog$ in point $(i)$ of the proposition, which additionally allows to prove a relation between the different terms of $z(\cdot;\rho)$, facilitating the analysis of the generalized directional derivatives later in the paper.
\begin{proposition}
\label{prop:feasibility_limit}
Let \cref{ass:bounded_level_sets} be satisfied,
$\set{\vecx_k}_{k\in\K_{\rho}}$ be a path-following subsequence, 
and $\K_{\rho}^{\tt x}\subseteq\K_{\rho}$ be the index set of an end-path subsequence with end-path point $\bar{\vecx}$. 
If $\cint(\cdot)$ is Lipschitz continuous near $\barx$, then for all $k\in \K_{\rho}^{\tt x}$ sufficiently large, the following holds
\begin{eqnarray}
\mbox{ \quad (i) } && \cint(\vecx_k + \vecd) < 0 \mbox{ for any direction } \vecd \in \R^n \mbox{ such that } \vecx_k+\vecd \in \F_k; \nonumber \\
\mbox{ \quad (ii) }  
&& f(\vecp_k) - f(\vecx_k) + \tfrac{1}{\rho_k}\big( \cext(\vecp_k) - \cext(\vecx_k) \big) \geq  \tfrac{\rho_k}{\cint(\vecp_k)}\left(  \cint(\vecp_k)-\cint(\vecx_k)\right)
\label{eq:delta_inequality}
\end{eqnarray}
where $\vecd_k$ is any poll direction in $\Dd_k$, 
and $\vecp_k = \vecx_k+ \vecd_k$ is the corresponding poll point.
\end{proposition}
\begin{proof}
\noindent (i) We first prove the result for $\phi^{\tt prox}(\vecx_k+\vecd)$ and then deduce that it also holds for $\cint(\vecx_k+\vecd)$. Indeed, for any $\vecy \in \mathbb{R}^n$, $\phi^{\tt prox}(\vecy) < 0$ if and only if $\cint(\vecy) < 0$, and $\phi^{\tt prox}(\vecy) = 0$ if and only if $\cint(\vecy) = 0$.
Let $\vecd \in \R^n$ be such that $\vecx_k+\vecd\in\F_k$. 
By the definition of $\F_k$, $\norm{\vecd} \leq \Delta_k$.
Since $\displaystyle\lim_{k\in\K_{\rho}^{\tt x}}\vecx_k = \barx$, and
$\displaystyle\lim_{ k\in\K_{\rho} }\Delta_k = 0$ (from~\cref{theo:convergence_delta}), it follows that
$$
\displaystyle\lim_{ k\in\K_{\rho}^{\tt x}} \phi^{\tt prox}(\vecx_k) \ = \ \displaystyle\lim_{k\in\K_{\rho}^{\tt x}} \phi^{\tt prox}(\vecx_k+\vecd)\  = \ \phi^{\tt prox}(\bar\vecx) \ \leq \ 0,
$$
where the last inequality follows from the fact that $\cint(\vecx_k)<0$ for all $k$, and, therefore, $\phi^{\tt prox}(\vecx_k)<0$ for all $k$. 
Equation (i) follows trivially when $\cint(\bar\vecx) < 0$, so we consider the situation where $\phi^{\tt prox} (\barx)=\cint(\bar\vecx) = 0$.
All iterations $k\in \K_{\rho}^{\tt x}$ are unsuccessful, 
and so  $\vecx_{k+1}=\vecx_k$,
$\Delta_{k+1}=\theta_{\Delta}\Delta_k$, with $\theta_{\Delta}\in (0,1)$,
and~\eqref{eq:penalty_criterion} ensures that $\Delta_{k+1} \leq \phi^{\tt prox}(\vecx_k)^2$
Since $\cint(\cdot)$ is Lipschitz near $\barx$, we deduce that  $\phi^{\tt prox}(\cdot)$
is also Lipschitz near $\barx$. Thus, there exists a constant $L_\phi >0$ such that for all $k\in\K_\rho^{\tt x}$ sufficiently large
{\small
$$|\phi^{\tt prox}(\vecx_k) - \phi^{\tt prox}(\vecx_k + \vecd)| 
\ \leq L_{\phi} \ \|\vecd\| 
\ \leq \ L_{\phi} \Delta_k
\ = \ \frac{L_\phi \Delta_{k+1}}{\theta_\Delta}
\ \leq \ \frac{L_\phi [\phi^{\tt prox}(\vecx_k)]^2}{\theta_\Delta}.
$$ }
It follows that
\begin{eqnarray*}
\phi^{\tt prox}(\vecx_k + \vecd) &\leq&  \frac{L_{\phi}}{\theta_{\Delta}} \phi^{\tt prox}(\vecx_k)^2 + \phi^{\tt prox}(\vecx_k)
\ =\  \phi^{\tt prox}(\vecx_k) \left(\frac{L_{\phi}}{\theta_{\Delta}} \phi^{\tt prox}(\vecx_k) + 1\right). 
\end{eqnarray*}

The right-hand-side is strictly negative when $k\in\K_{\tt \rho}^{\tt x}$ is large, since
{\small
$\phi^{\tt prox}(\vecx_k) < 0 = \phi^{\tt prox}(\bar\vecx)$, $L_\phi>0$}, and $\theta_\Delta>0$.
Thus, the entire frame is strictly feasible with respect to the constraints treated by the logarithmic barrier, which proves (i).

\noindent (ii) Consider $k\in\K_{\tt \rho}^{\tt x}$, $\vecd_k\in\Dd_k$ and set $\vecp_k = \vecx_k + \vecd_k$.
Equation~\eqref{eq:uns_zeta} can be rewritten as
\begin{equation*}
f(\vecp_k) - f(\vecx_k) + \tfrac{1}{\rho_k}\big( \cext(\vecp_k) - \cext(\vecx_k) \big) + \rho_k\log\left(\frac{-\cint(\vecx_k)}{-\cint(\vecp_k)}
\right) \ \geq \ 0.
\end{equation*}
Using the fact that $\cint(\vecx_k)< 0$, that (i) holds for $\vecd = \vecd_k$ when $k$ is large, and since $\log(t)\leq t-1$ for all $t>0$, it follows that
\begin{equation*}
\log\left(\frac{-\cint(\vecx_k)}{-\cint(\vecp_k)}\right) \ \leq \ \dfrac{1}{-\cint(\vecp_k)}\left(  \cint(\vecp_k)-\cint(\vecx_k)\right).
\end{equation*}
Combining the two inequalities concludes the proof.
\end{proof}
For a function \(f:\R^n\to\overline{\R}\), the Clarke generalized directional derivative~\cite{Clarke_1990} of $f$ at $\vecx \in \R^n$ along a direction $\bard \in \R^n$  is given by
$$f^\circ(\vecx;\bard) \ = \ \limsup_{
\substack{\vecy\to \vecx\\ t\searrow 0}
}\, \frac{f(\vecy+t\bard) - f(\vecy)}{t}.$$

The generalized derivative proposed by Clarke is a powerful tool for the analysis of nonsmooth optimization problems. 
When the function $f$ is Lipschitz near a point $\vecx$, then the superior limit exists and is finite, see~\cite[Chapter 2]{Clarke_1990}.
The following technical result~\cite[Proposition 3.9]{AuDe2006} is used in the theoretical analysis. 
\begin{proposition}
\label{prop:dk_to_d}
Let $f:\R^n\to\R$ be Lipschitz continuous near $\vecx$. Then
\begin{equation}
\label{eq:dk_to_d}
f^\circ(\vecx,\vecd)= \limsup_{
\substack{\vecy\to \vecx \\ t\searrow 0 \\ \vecw\to\vecd}
}\, \frac{f(\vecy+t\vecw) - f(\vecy)}{t} \mbox{ .}
\end{equation}
\end{proposition}

The next definition introduces formally the \textit{path-refining} directions. 
These are similar to the \textit{refining} directions introduced in~\cite{AuDe2006}, with the difference that the behaviour of the values of $\cint(\cdot)$ is considered.

\begin{definition}\label{def:path_refining}
Let $\set{\vecx_k}_{k\in\K_{\rho}}$ be a path-following subsequence, 
and $\K_{\rho}^{\tt x}\subseteq\K_{\rho}$ be the index set of an end-path subsequence with end-path point $\bar{\vecx}$ such that the normalized sequence of poll directions $\{\vecd_k\} \subset \Dd_k$ converges to some direction $\bard \in \R^n$, i.e.,
$\displaystyle\lim_{k\in\K_\rho^{\tt x}}\dfrac{\vecd_k}{\norm{\vecd_k}} \ = \ \bard$.
Then, $\bard$ is said to be a path-refining direction for $\bar{\vecx}$, if
\begin{equation}
\label{eq:c_pathrefining_d}
\cint(\vecx_k+\vecd_k)\ \leq\ \cint(\vecx_k) \quad \mbox{ for all } k \in \K_\rho^{\tt x},
\end{equation}
\end{definition}
\begin{remark}\label{remark:negative_clarke_path_refining}
Suppose that \(c^\circ(\vecx;\vecd)<0\).  
Then, by the definition of the Clarke directional derivative,  
for any sequences \(\{\vecx_k\}\to\vecx\) and \(\{\vecd_k\}\) satisfying  
\[
\lim_{k \to \infty}\frac{\vecd_k}{\|\vecd_k\|} \;=\; \vecd
\quad\text{and}\quad
\|\vecd_k\|\to 0,
\]
there exists \(k_0\) such that
$c(\vecx_k+\vecd_k) \le c(\vecx_k)$
for all $k\ge k_0$.
Thus, setting \(\K_\rho^{\tt x} := \{k : k \ge k_0\}\) implies that  
the vector \(\vecd\) is a path-refining direction for \(\vecx\).
\end{remark}

It is possible to derive a property concerning the violation function $\cext(\cdot)$, defined in~\eqref{eq:violation_function}, along path-refining directions. Such a property is presented and proved in the following proposition.
\begin{proposition}\label{prop:v_clarke_nonnegative}
Let \cref{ass:bounded_level_sets} be satisfied, $\set{\vecx_k}_{k\in\K_{\rho}}$ be a path-following subsequence, with end-path point $\bar{\vecx}$, and let $\vecd$ be a path-refining direction. If $f, \cint$, and $\cext$ are Lipschitz continuous near $\barx$, then
\begin{equation}\label{eq:v_clarke_nonnegative}
(\cext)^\circ(\barx;\vecd) \ \geq \ 0.
\end{equation}
\end{proposition}
\begin{proof}
By the definition of path-refining direction, a set $\K_\rho^{\tt x}$ exists such that {\small $\displaystyle\lim_{k\in\K_\rho^{\tt x}}\dfrac{\vecd_k}{\norm{\vecd_k}} = \bard$} and~\eqref{eq:c_pathrefining_d} hold.
\Cref{prop:feasibility_limit} ensures $\cint(\vecx_k+\vecd_k)<0$ for all $\vecd_k\in\Dd_k$ and $k\in\K_\rho^{\tt x}$ sufficiently large. Recalling $\rho_k>0$ for all $k$, and using~\eqref{eq:c_pathrefining_d}, it follows for $k\in\K_\rho^{\tt x}$ sufficiently large
\begin{equation*}
\tfrac{\rho_k}{\cint(\vecx_k+ \vecd_k)}\left(  \cint(\vecx_k+\vecd_k)-\cint(\vecx_k)\right) \ \geq \ 0.
\end{equation*}
Thus, using~\eqref{eq:delta_inequality} from \cref{prop:feasibility_limit}, for sufficiently large $k\in\K_\rho^{\tt x}$, it holds
\begin{equation*}
f(\vecx_k+\vecd_k) - f(\vecx_k)  + \frac{1}{\rho_k}\big( \cext(\vecx_k+\vecd_k) - \cext(\vecx_k) \big) \geq 0.
\end{equation*}
Dividing both sides of the latter equation by $\norm{\vecd_k}$, multiplying by $\rho_k$ and taking the superior limit for $k\in\K_\rho^{\tt x}$
\begin{equation}\label{eq:lim_sup_v_1}
\limsup_{ k\in\K_\rho^{\tt x}} \left[ \rho_k\dfrac{f(\vecx_k+ \vecd_k) - f(\vecx_k)}{\norm{\vecd_k}} + \dfrac{\cext(\vecx_k+ \vecd_k)-\cext(\vecx_k)}{\norm{\vecd_k}}\right] \geq 0.
\end{equation}
Since $f(\cdot)$ is Lipschitz around $\barx$, then the sequence {\small $\displaystyle \left\{\dfrac{f(\vecx_k+\vecd_k) - f(\vecx_k)}{\norm{\vecd_k}}\right\}$} is bounded. Recalling that, by \cref{theo:convergence_tau},  $\lim_{k\to\infty} \rho_k=0$, we can write
\begin{equation*}
\limsup_{k\in\K_\rho^{\tt x}} \dfrac{\cext(\vecx_k+ \vecd_k)-\cext(\vecx_k)}{\norm{\vecd_k}} \geq 0.
\end{equation*}
Using the fact that $\cext(\cdot)$ is Lipschitz near $\barx$, by \cref{prop:dk_to_d}, we can write
\begin{equation*}
(\cext)^\circ(\barx;\vecd) \geq \limsup_{ k\in\K_\rho^{\tt x}} \dfrac{\cext(\vecx_k+ \vecd_k)-\cext(\vecx_k)}{\norm{\vecd_k}} \geq 0,
\end{equation*}
proving~\eqref{eq:v_clarke_nonnegative}, and concluding the proof.
\end{proof}
The result of \cref{prop:v_clarke_nonnegative} is used to define conditions to ensure feasibility for the end-path points with respect to Problem~\eqref{P0}. Note that any sequence $\set{\vecx_k}$ of incumbent points generated by the algorithm is such that $\cint(\vecx_k)<0$ for all $k$, so that any limit point $\barx$ of the sequence belongs to $\Omegalog$. Therefore, in order to show that $\barx$ belongs to $\Omega$, one has to show that $\cext(\barx)=0$. Since $\cext(\vecx)\geq 0$ for all $\vecx\in\R^n$, any point $\barx$ such that $\cext(\barx)=0$ is a global minimum of $\cext(\cdot)$, i.e., for any $\epsilon>0$, any $\vecy\in\B(\barx,\epsilon)$ is such that $\cext(\vecy)\geq \cext(\barx)$, where $\B(\barx,\epsilon)$ is a ball centred at $\barx$. It trivially follows that $(\cext)^\circ(\barx;\vecd)\geq 0$ for all $\vecd\in\R^n$. If, on the other hand, $\cext(\barx)>0$, there might exist directions $\vecd\in\R^n$ such that $(\cext)^\circ(\barx;\vecd)<0$.
In view of \cref{prop:v_clarke_nonnegative}, such directions cannot be path-refining directions.
\subsection{Feasibility}
In this subsection, results involving all the introduced objects are shown, which allow to define conditions under which the end-path points generated by \cref{LOG-MADS-1} are feasible for Problem~\eqref{P0}.

The following provides the definition of the hypertangent cone~\cite{Clarke_1990}.

\begin{definition}[Hypertangent Cone]
\label{def:hypertangent_cone}
A vector $\vecd\in\R^n$ is said to be hypertangent to $\Omega$ at $\vecx$ if there exist $\epsilon>0$ such that
\begin{equation*}
\vecy+t\vecw\in\Omega \text{ for all }\vecy\in\Omega\cap \B(\vecx,\epsilon), \quad \vecw\in \B(\vecd,\epsilon), \quad t\in(0,\epsilon).
\end{equation*}
\end{definition}
The set of all hypertangent vectors to $\Omega$ at $\vecx$ is called the hypertangent cone, denoted as $\T^H_{\Omega}(\vecx)$. The hypertangent cone appears in a necessary optimality condition for a solution $\vecx\in\Omega$ to be a local minimizer of Problem~\eqref{P0}, for more details, see~\cite[Theorem~6.10]{AuHa2017}. In fact, if $f$ is Lipschitz continuous near a local minimizer $\vecx\in\Omega$ for Problem~\eqref{P0}, then $f^\circ(\vecx;\vecd)\geq 0$ for every direction $\vecd\in \T^H_\Omega(\bar{\vecx})$.

The next definition introduces a first constraint qualification, which we will use in the analysis of feasibility.

\begin{definition}[Feasibility Constraint Qualification]
\label{cq:fcs} 
A point $\vecx\in\R^n$  is said to satisfy the feasibility constraint qualification  (FCQ) for problem~\eqref{P0}, if the following conditions hold:
\begin{itemize}
\item[(a)]
$\cint(\cdot)$ and $\cext(\cdot)$ are Lipschitz continuous near $\vecx$;

\item[(b)] 
either $\vecx \in \Omega$,  or $\vecx \in \Omegalog\setminus\Omegaext$ and there exists a direction $\vecd \in \R^n$ such that $(\cext)^\circ(\vecx; \vecd) < 0$ and
\begin{equation}\label{eq:cq_notfeasible_fcq}
(\cint)^\circ(\vecx; \vecd) < 0  \quad \text{ if } \cint(\vecx) = 0.
\end{equation}
\end{itemize}
\end{definition}
The FCQ, as introduced in \cref{cq:fcs}, is designed to ensure that, at any candidate limit point $\vecx$, the local geometry of the constraint system is sufficiently regular to guarantee that the algorithmic path can approach feasibility. This is particularly important in the nonsmooth, nonconvex setting, where classical constraints qualifications such as MFCQ~\cite{Mangasarian67},  may not apply or may fail to capture the subtleties of the problem structure. At its core, the FCQ requires the point to be feasible with respect to the exterior constraints. Otherwise, for points that are infeasible with respect to the exterior violation function, i.e., when $\cext(\vecx) > 0$, additional conditions are required. When using a penalty approach to deal with nonlinear constraints, the problem of feasibility can be translated as the minimization of the norm of constraint violations, e.g. $v(\vecx) = \sum_{\ell=1}^m [g_\ell^{+}(\vecx)]^2 + \sum_{j=1}^p [h(\vecx)]^2$. When mixing the penalty approach with an interior violation function, the latter translation does not properly reflect the constraint-handling strategy as the merit function is not defined for all $\vecx\in\R^n$. A coherent representation of the feasibility problem when using the proposed PIP strategy would be
\begin{align}
\displaystyle\min_{\vecx \in \mathbb{R}^n} \quad & \cext(\vecx) \label{P_feas}\tag{$P_{feas}$}\\
\text{s.t.} \quad & \cint(\vecx)\leq 0.\nonumber
\end{align}
The feasible region of~\eqref{P_feas} is $\Omegalog$. The optimization algorithm may stall at points that are stationary for~\eqref{P_feas}, thus, points such that $(\cext)^\circ(\vecx;\vecd)\geq 0$ for all $\vecd\in\T^H(\vecx,\Omegalog)$. By definition of hypertangent cone, and by definition of Clarke generalized derivative, it follows trivially that if a direction $\vecd$ is such that $(\cint)^\circ(\vecx,\vecd)<0$, as in~\eqref{eq:cq_notfeasible_fcq}, then $\vecd$ belongs to $\vecd\in\T^H(\vecx,\Omegalog)$. Therefore, FCQ ensures that points are not stationary for~\eqref{P_feas}.

For our asymptotic analysis, it is well known that, for direct search methods, a density assumption on the sequence of sets of poll directions ${\mathcal{D}_k}$~\cite{AuDe2006} is needed. Let us now formalize the notion of density over the unit sphere.

\begin{definition}
\label{def:dense_directions}
Let $\K$ be an infinite subset of indices.
A sequence of sets of directions $\set{\Dd_k}_{k\in \K}$ is said to be dense in the unit sphere $\B(\bvec{0}_n,1)$ if, for any $\bard\in \B(\bvec{0}_n,1)$, there exists a sequence $\set{\vecd_k}_{k\in \K^{\tt d}}$, with $\vecd_k\in\Dd_k$ and $\K^{\tt d}\subseteq\K$ such that
$\displaystyle\lim_{
\substack{k\to +\infty \\ k\in \K^{\vecd}}
}\dfrac{\vecd_k}{\norm{\vecd_k}} = \bard$.
\end{definition}
In what follows, we show that any end-path point satisfying the FCQ  (as in~\cref{cq:fcs}) is feasible for~\eqref{P0}.

\begin{theorem}
\label{th:feas}
Let \cref{ass:bounded_level_sets} hold and  $\set{\vecx_k}$ be the sequence generated by \cref{LOG-MADS-1}. Let $\set{\vecx_k}_{k\in\K_{\rho}}$ be a path-following subsequence of $\set{\vecx_k}$ and consider $\K_{\rho}^{\tt x}\subseteq\K_{\rho}$ to be the index set of an end-path subsequence with end-path point $\bar{\vecx}$. 
Assume that $f$ is Lipschitz continuous near $\bar{\vecx}$, that the sequence of sets of poll directions $\{\Dd_k\}$ is dense in the unit sphere and that $\bar{\vecx}$ satisfies the FCQ. Then, $\bar{\vecx}$ is feasible for  problem~\eqref{P0}.
\end{theorem}
\begin{proof}
By definition of the iterates $\vecx_k$ generated by \cref{LOG-MADS-1}, it follows that  $\cint(\vecx_k) < 0$ for all $k$, and continuity implies that $\cint(\bar{\vecx}) \leq 0$ (i.e., $\vecx \in \Omegalog$). Thus, under the assumption that $\bar{\vecx}$ satisfies the FCQ,  it suffices to prove that $\bar{\vecx} \in \Omegaext$ in order to deduce that $\bar{\vecx}$ is feasible for problem~\eqref{P0}. 
The analysis is divided into two cases.

\noindent \underline{\textbf{Case (i)} $\cint(\bar{\vecx}) < 0$:} Since the sets $\{\Dd_k\}$ are dense in the unit sphere, 
then for any $\vecd \in \B(\bvec{0}_n,1)$, there exists a subsequence of indices
$\K_{\rho}^{\tt d}\subseteq \K_{\rho}^{\tt x}$ such that
$\displaystyle\lim_{\substack{k\to +\infty \\ k\in \K_{\rho}^{\tt d}}} \dfrac{\vecd_k}{\norm{\vecd_k}} = \vecd$.
In the following, we use $\vecp_k = \vecx_k+\vecd_k$. Hence, by \cref{prop:feasibility_limit}, we have, for sufficiently large $k \in \K_{\rho}^{\tt d}$,
\begin{equation*}
f(\vecp_k) - f(\vecx_k) + \tfrac{1}{\rho_k}\big( \cext(\vecp_k) - \cext(\vecx_k) \big) \geq  \tfrac{\rho_k}{\cint(\vecp_k)}\left(  \cint(\vecp_k)-\cint(\vecx_k)\right).
\end{equation*}
Hence,
\begin{equation*}
\rho_k \frac{f(\vecp_k) - f(\vecx_k)}{\|\vecd_k\|} +
\frac{\cext(\vecp_k) - \cext(\vecx_k)}{\|\vecd_k\|}  \geq  \tfrac{\rho_k^2}{\cint(\vecp_k)}\left(  \frac{\cint(\vecp_k) - \cint(\vecx_k)}{\|\vecd_k\|}\right).
\end{equation*}
Taking the superior limit for $k\in\K_{\rho}^{\tt x}$, we obtain:
{\small
\begin{equation}\label{eq:feasible_endpath_1}
\limsup_{\substack{k \to \infty \\ k \in \K_{\rho}^{\tt d}}}  
\rho_k \frac{f(\vecp_k) - f(\vecx_k)}{\|\vecd_k\|} +
\frac{\cext(\vecp_k) - \cext(\vecx_k)}{\|\vecd_k\|} 
\geq 
\limsup_{\substack{k \to \infty \\ k \in \K_{\rho}^{\tt d}}} 
\tfrac{\rho_k^2}{\cint(\vecp_k)}\left(  \frac{\cint(\vecp_k) - \cint(\vecx_k)}{\|\vecd_k\|}\right).
\end{equation}
}
Since $\bar{\vecx}$ is an end-path point, we have $\set{\vecx_k}_{k\in\K_{\rho}^{\tt d}}\to\barx$. Also, since $\|\vecd_k\| \leq \Delta_k$ and  $\set{\Delta_k}_{k\in\K_{\rho}^{\tt d}} \to 0$ (by~\cref{theo:convergence_delta}), it follows that $\set{\vecp_k}_{k\in\K_\rho^{\tt x}} \to \bar{\vecx}$.
Using $\set{\rho_k}_{k\in\K_\rho} \to 0$, $\cint(\bar{\vecx}) < 0$, and the Lipschitz continuity of $\cint(\cdot)$ near $\barx$ (as it satisfies the FCQ), we can write
$$
\lim_{k \in \K_{\rho}^{\tt d}} \frac{\rho_k^2}{\cint(\vecp_k)} = 0 \quad \text{and} \quad \left\{\frac{\cint(\vecp_k) - \cint(\vecx_k)}{\|\vecd_k\|}\right\} \text{ is bounded}.
$$
Hence, the right-hand side of~\eqref{eq:feasible_endpath_1} equals zero. \\ [1.em]
From $\set{\rho_k}_{k\in\K_\rho} \to 0$, and Lipschitz continuity of $f(\cdot)$ near $\barx$, it follows that
{\small
$$
\limsup_{\substack{k \to \infty \\ k \in \K_{\rho}^{\tt d}}} 
\left[ \rho_k \frac{f(\vecp_k) - f(\vecx_k)}{\|\vecd_k\|} +
\frac{\cext(\vecp_k) - \cext(\vecx_k)}{\|\vecd_k\|} \right]
\ =\ 
\limsup_{\substack{k \to \infty \\ k \in \K_{\rho}^{\tt d}}}
\frac{\cext(\vecp_k) - \cext(\vecx_k)}{\|\vecd_k\|} \ \geq \ 0.
$$
}
This implies that $(\cext)^\circ(\bar{\vecx};\vecd) \geq 0$ for all $\vecd \in \B(\bvec{0}_n,1)$, which leads to 
\begin{equation}\label{eq:v_nonnegative_proof_feasible}
(\cext)^\circ(\bar{\vecx};\vecd) \geq 0, \quad \forall~\vecd \in \R^{n}.
\end{equation}

Since $\bar{\vecx}$ satisfies the FCQ, if $\cext(\bar{\vecx}) > 0$, then $\bar{\vecx} \notin \Omegaext$, and there would exist a direction $\vecd \in \B(\bvec{0}_n,1)$ such that $(\cext)^\circ(\bar{\vecx};\vecd) < 0$, which contradicts~\eqref{eq:v_nonnegative_proof_feasible}. Therefore, $\cext(\bar{\vecx}) = 0$, and $\bar{\vecx}$ is feasible for problem~\eqref{P0}.

\smallskip
\noindent\underline{\textbf{Case (ii)} $\cint(\bar{\vecx}) = 0$:}
By contradiction, assume that $\bar{\vecx} \notin \Omegaext$. 
Since $\bar{\vecx}$ satisfies the FCQ, then there exists a direction $\vecd$ such that 
$(\cint)^\circ(\bar{\vecx}; \vecd) < 0$ and $(\cext)^\circ(\bar{\vecx}; \vecd) < 0$. 

Because $(\cint)^\circ(\bar{\vecx}; \vecd) < 0$, \cref{remark:negative_clarke_path_refining} 
implies that $\vecd$ is a path-refining direction. 
Hence, by \cref{prop:v_clarke_nonnegative}, we obtain 
$(\cext)^\circ(\bar{\vecx}; \vecd) \geq 0$ which contradicts 
$(\cext)^\circ(\bar{\vecx}; \vecd) < 0$. 
Therefore, $\bar{\vecx} \in \Omegaext$, and thus 
$\bar{\vecx}$ is feasible for problem~\eqref{P0}.

\end{proof}

\subsection{Stationarity}
\label{sec:stationarity}
This subsection uses the Clarke subdifferential, a generalization of the gradient for non-differentiable functions, to provide stationarity properties on the sequence of iterates generated by \cref{LOG-MADS-1}. 
First, a notion of stationarity for nonsmooth constrained problems is introduced. Then a constraint qualification is presented under which the end-path points are proved to be stationary. 

The stationarity results hold when the problem does not present equality constraints. 
This is not surprising as the additional theoretical complexity caused by the presence of nonsmooth equality constraints is well-known for mathematical programming problems, especially when barrier or penalty techniques are used to solve them~\cite{DiPillo95,Polak83}.
Stationarity results of the Penalty-Interior Point Method 
in presence of equality constraints are proposed  in~\cite{AuLedPey2015,BrLiLu2025,BrCuLiSi2024}.

The Clarke subdifferential of a function is defined as follows.
\begin{definition}[Clarke's  Subdifferential~\cite{Clarke_1990}]\label{def:subdifferential}
For any $\vecx\in\R^n$, Clarke's subdifferential of $f:\R^n\to\overline\R$ at $\vecx$ is defined as: $$\partial f(\vecx):= \set{\vecv\in\R^n\ \mid\ f^\circ(\vecx;\vecd)\geq \vecv^\top \vecd,\ \text{ for all } \vecd\in\R^n}.$$
\end{definition}

When $f(\cdot)$ is Lipschitz near a point $\vecx$, then $\partial f(\vecx)$ is nonempty and compact.

We provide a constraint qualification variant, similar to~\cite[Assumption~A3]{AuDe09a}, for our first-order stationarity  analysis. 

\begin{definition}[Stationarity Constraint Qualification]
\label{cq:scq} 

A point $\vecx\in\Omega$  is said to satisfy the stationarity constraint qualification  (SCQ) for problem~\eqref{P0} with $p=0$, if the following conditions hold:
\begin{itemize}
\item[(a)]
$\cint(\cdot)$ and $\cext(\cdot)$ are Lipschitz continuous near $\vecx$;

\item[(b)] \(\forall ~\vecd \in \T^H_\Omega(\vecx)\), $(\cint)^\circ(\vecx;\vecd) < 0$; 
\item[(c)] \(\forall ~\vecd \in \T^H_\Omega(\vecx)\), there exists \(\epsilon > 0\) such that $(\cext)^\circ(\vecy;\vecd) < 0$, for any   $\vecy \in \B(\vecx,\epsilon) \setminus \Omegaext$.
\end{itemize}
\end{definition}
\Cref{cq:scq} requires that at $\vecx$, the Clarke directional derivative of $\cint$ in any hypertangent direction is strictly negative;
and for the exterior constraints, the Clarke directional derivative of $\cext$ is strictly negative for all points outside $\Omegaext$ in a neighbourhood around $\vecx$.

As the algorithm progresses and the penalty/barrier parameters are driven to their limits, any accumulation point (end-path point) must be feasible with respect to the original constraints. The SCQ rules out pathological cases where the algorithm converges to a point on the boundary of the feasible set, without generating strictly feasible points, due to the lack of descent directions for the constraint violations
(e.g.,~\cite[Appendix A]{AuDe09a}).

The SCQ ensures that, locally, the feasible region is not ``tangentially degenerate''—that is, there are always directions that can strictly decrease any active constraint violation.
This is analogous to the classical MFCQ, but adapted to the nonsmooth, aggregated, and split-barrier/penalty structure of our problem.
The SCQ is incompatible with equality constraints, as it implicitly implies that $\T^H_\Omega(\vecx)$ is nonempty: 
the hypertangent cone is necessarily empty
when the dimension of the feasible region is strictly less than $n$.

The next result is not related directly to \cref{LOG-MADS-1}, but it is useful for our stationarity analysis.
\begin{proposition}[{\cite[Theorem 2.3.7]{Clarke_1990}}] \label{prop:Lebourg}
Let $\vecx,\vecy\in\R^n$ and let $c:\R^n\to\R$ be Lipschitz continuous in an open set containing the line segment $[\vecx,\vecy]$. Then, there exists $\bvec{u}\in (\vecx,\vecy)$ and $\vecgreek{\xi}\in \partial c(\bvec{u})$ such that
\begin{equation}\label{eq:Lebourg}
c(\vecx)-c(\vecy)\;=\;\vecgreek{\xi}^\top(\vecx-\vecy).
\end{equation}
\end{proposition}

The following proposition is similar to a result with the progressive barrier~\cite{AuDe2009} 
technique used to handle inequality constraints.

\begin{proposition}
\label{lemma:nonincreas_constraints}
Let $\vecx\in\Omega$ be a point that satisfies the SCQ and let $\vecd\in \T^H_\Omega(\vecx)$. Then, there exists $\epsilon>0$ such that, for any $\vecy\in \B(\vecx,\epsilon)$ and $\vecw\in \B(\vecd,\epsilon)$, the following holds for all $t\in(0,\epsilon)$
\begin{subequations}
\begin{align}
\label{eq:decrease_cq_int}
\cint(\vecy+t\vecw) \leq \cint(\vecy),\\
\label{eq:decrease_cq_ext}
\cext(\vecy+t\vecw) \leq \cext(\vecy).
\end{align}
\end{subequations}
\end{proposition}

\begin{proof}
Throughout, $\cint$ and $\cext$ are locally Lipschitz near $\vecx$ by condition (a) of the SCQ, hence Clarke’s generalized directional derivatives
$(\cint)^\circ(\cdot;\cdot)$ and $(\cext)^\circ(\cdot;\cdot)$ are well-defined.
The proof of~\eqref{eq:decrease_cq_int} follows using the Lipschitz continuity of $\cint$ near $\vecx$ and condition (b) of SCQ, i.e., $(\cint)^\circ(\vecx,\vecd)<0$ for all $\vecd\in\T^H_\Omega(\vecx)$.
Regarding~\eqref{eq:decrease_cq_ext}, let us suppose, by contradiction, that for any $\epsilon>0$, there exist $\vecy_\epsilon\in \B(\vecx,\epsilon)$, $\vecw_\epsilon \in \B(\vecd,\epsilon)$ and $t_\epsilon \in (0,\epsilon)$ such that
\begin{equation} \label{eq1:lem32}
\cext(\vecy_\epsilon+t_\epsilon \vecw_\epsilon) > \cext(\vecy_\epsilon).
\end{equation}
\noindent If $\cext(\vecy_\epsilon)=0$, since $\vecd\in\T^H_\Omega(\vecx)$, by the definition of $\T^H_\Omega(\vecx)$, for small enough $\epsilon$ it holds $\vecy_\epsilon+t_\epsilon\vecw_\epsilon\in\Omega$, which implies $\cext(\vecy_\epsilon+t_\epsilon\vecw_\epsilon)=0$, and hence contradicts \eqref{eq1:lem32}. 

\noindent If $\cext(\vecy_\epsilon)>0$, since $\T^H_\Omega(\vecx)$ is an open convex set, for small enough $\epsilon$, the vector $\vecw_\epsilon\in\T^H_\Omega(\vecx)$, and the line segment $ [\vecy_\epsilon,\vecy_\epsilon+t_\epsilon\vecw_\epsilon]\subset\B(\vecx,\epsilon)$. Therefore, by condition (c) of SCQ, $(\cext)^\circ(\bvec{u};\vecw_\epsilon)<0$ for all $\bvec{u} \in [\vecy_\epsilon,\vecy_\epsilon+t_\epsilon\vecw_\epsilon]$.
Using \cref{prop:Lebourg} and \eqref{eq1:lem32}, there exists $\bvec{u}_\epsilon \in[\vecy_\epsilon,\vecy_\epsilon+t_\epsilon\vecw_\epsilon]$ and $\vecgreek{\xi}_{\epsilon}\in\partial \cext(\bvec{u}_\epsilon )$ such that $$ t_\epsilon\vecw_\epsilon^\top\vecgreek{\xi_{\epsilon}} =\cext(\vecy_\epsilon+t_\epsilon\vecw_\epsilon)-\cext(\vecy_\epsilon)>0.$$ 
Hence, by \cref{def:subdifferential} of $\partial \cext(\bvec{u}_{\epsilon})$, it follows that $(\cext)^\circ(\bvec{u};\vecw_\epsilon)\geq \vecw_\epsilon^\top\vecgreek{\xi}_{\epsilon} > 0$, again, a contradiction. Hence,~\eqref{eq:decrease_cq_ext} holds.
\end{proof}

Next is our final result.

\begin{theorem}
Consider problem~\eqref{P0} with $p=0$. Let \cref{ass:bounded_level_sets} hold and  $\set{\vecx_k}$ be the sequence of iterates generated by \cref{LOG-MADS-1}. Let $\set{\vecx_k}_{\K_{\rho}}$ be a path-following subsequence of $\set{\vecx_k}$ and consider $\K_{\rho}^{\tt x} \subseteq \K_{\rho}$ to be the index set of a relative end-path subsequence, with end-path point $\bar{\vecx}$. Assume that $f$ is Lipschitz continuous around $\bar{\vecx}$, that the sequence of sets of poll directions $\{\Dd_k\}$ is dense in the unit sphere, and that $\bar{\vecx}$ satisfies the SCQ.  Then, $f^\circ(\barx;\vecd) \geq 0$ for all $\vecd\in\T^H_\Omega(\barx)$.
\end{theorem}
\begin{proof}
Since the sets $\{\Dd_k\}$ are dense in the unit sphere, for any $\vecd \in \B(\bvec{0}_n,1)$, there exists a subsequence of $\K_{\rho}^{\tt d} \subseteq K_{\rho}^{\tt x}$ such that
$\displaystyle\lim_{\substack{k\to +\infty \\ k\in \K^{\vecd}}}\tfrac{\vecd_k}{\norm{\vecd_k}} = \vecd$. In the following, we use $\vecp_k = \vecx_k+\vecd_k$. 
\cref{prop:feasibility_limit} ensures that for sufficiently large $k \in \K_{\rho}^{\tt d}$,
\begin{equation}\label{eq:1:thm:SCQ}
\dfrac{f(\vecp_k) - f(\vecx_k)}{\|\vecd_k\|} \ \geq \ \dfrac{\rho_k}{\cint(\vecp_k)}\left(  \dfrac{\cint(\vecp_k)-\cint(\vecx_k)}{\|\vecd_k\|}\right) - \frac{1}{\rho_k}\left( \dfrac{\cext(\vecp_k) - \cext(\vecx_k)}{\|\vecd_k\|} \right).
\end{equation}
Since $\{\vecx_k\}_{k\in \K_\rho^{\tt d}}\to\barx$, and, for all $\vecd\in\T^H_\Omega(\barx)$, there exists a subsequence of $\overline{\K}_\rho^{\tt d} \subseteq \K_\rho^{\tt d}$ such that $\left\{\tfrac{\vecd_k}{\|\vecd_k\|}\right\}_{k\in \overline{\K}_\rho^{\tt d} }\to\vecd$. Thus, as $\barx$ satisfies SCQ,  \cref{lemma:nonincreas_constraints} ensures that~\eqref{eq:decrease_cq_ext} and~\eqref{eq:decrease_cq_int} hold by setting $\vecw=\vecd_k/\|\vecd_k\|$, $t=\|\vecd_k\|$ and $\vecy=\vecx_k$ for $k\in\overline{\K}_{\rho}^{\tt d}$ sufficiently large. Hence, for all $\vecd\in\T^H_\Omega(\barx)$ and sufficiently large $k \in \overline{\K}_\rho^{\tt d}$,
\begin{equation*}
-(\cext(\vecp_k)-\cext(\vecx_k))\geq 0 \quad \text{and} \quad
\dfrac{\cint(\vecp_k)-\cint(\vecx_k)}{\cint(\vecp_k)}\geq 0,
\end{equation*} 
which implies that \eqref{eq:1:thm:SCQ} leads to $\frac{f(\vecp_k)-f(\vecx_k)}{\|\vecd_k\|} \geq 0$ for all $\vecd\in\T^H_\Omega(\barx)$ and $k \in \overline{\K}_\rho^{\tt d}$ sufficiently large. Hence, we deduce that $f^\circ(\barx;\vecd) \geq 0$ for all $\vecd\in\T^H_\Omega(\barx)$.
\end{proof}

\section{Practical implementation}
\label{sec-Impl}

This section details some practical implementation aspects.

\subsection{Coding algorithm}
\label{subsection::coding}
All MADS variants in this work are implemented in {\sf NOMAD~4}~\cite{nomad4paper} with the default granular mesh called GMesh. In this variant of the mesh, the frame size $\Delta_k$ is updated using \textbf{increase} or \textbf{decrease} functions instead of being multiplied by $\theta_\Delta$ or $\theta_\Delta^{-1}$. The convergence properties of MADS with GMesh, the updating functions and the initialization of $\Delta_0$ are detailed in~\cite{AuLeDTr2018}.

The MADS Penalty-Interior Point (MADS-PIP) algorithm is coded in a development version forked from release~4.5.2\footnote{\url{https://github.com/bbopt/nomad}}. In the code, the MADS-PIP algorithm class derives from the MADS progressive barrier (MADS-PB) algorithm class in which the objective function is changed, the infeasibility measure is ignored and the condition of {\bf \blue Step~3} is tested at each unsuccessful iteration. 
Coding this new algorithm relies on the {\sf NOMAD~4} capability to dynamically change at each iteration the computation of the objective function and infeasibility measure from all blackbox problem outputs. Once the partition sets $\Glog$ and $\Gext$ have been identified, the objective function values can be computed as the merit function $z(\vecx_k,\rho_k)$ of \cref{merit-z} and the infeasibility measure can be fixed to zero. The penalty parameter $\rho$ and the partition sets can be updated periodically as detailed in the next section.

\subsection{Switches and numerical choices}
\label{subsection::switch}

As suggested in \cref{sec-algo}, the PIP strategy can benefit from dynamical updates of the constraints index partition sets $\Glog$ and $\Gext$. We propose a slight modification of {\bf \blue Step~4} of \cref{LOG-MADS-1}, as follows.

\begin{algorithm}
\begin{algorithmic}[0]
\LComment{{\bf Step~4}: Stopping criterion check}
\If{some stopping criterion is satisfied} \State{\textbf{Stop}} \Else 
\State Set 
$\Glog \gets \Glog\cup\{\ell\}$ and  $\Gext \gets \Gext\setminus\{\ell\}$
for each $\ell\in\Gext$ for which $g_\ell(\vecx_{k+1})<0$
\State $k \gets k+1 $ and  go to {\bf \blue Step~0}
\EndIf
\end{algorithmic}
\end{algorithm}

The update in {\bf \blue Step~4} consists of verifying whether the inequality constraints contributing to the exterior violation function are strictly feasible at the new incumbent solution. That is, at each iteration $k$, the algorithm checks the condition $g_\ell(\vecx_{k+1})<0$ for each $\ell\in\Gext$. If the latter holds for an index $\bar{\ell}$, then $\bar{\ell}$ is removed from $\Gext$ and assigned to $\Glog$, hence, the inequality constraint is moved into the interior violation function. Since the number of inequality constraints is finite, the update in {\bf \blue Step~4} can be performed only a finite number of times. When the last of such updates occurs, the algorithm follows the theoretical analysis of the previous section. Finally, note that at unsuccessful iterations it holds $\vecx_{k+1}=\vecx_{k}$; therefore, the constraint index partition sets can only change at successful iterations. It follows that the penalty parameter and the constraint index partition sets cannot be updated during the same iteration.

In the computational experiments, the rule in the new step of the algorithm has been implemented considering a tolerance, that is
$$ g_\ell(\vecx_{k+1}) \leq -\varepsilon^{\tt ext},$$
where $\varepsilon^{\tt ext}$ was set equal to $10^{-14}$. The tolerance allows for a choice that is more robust to cases where the points may be extremely close to the boundary of the feasible region and machine precision might cause inconsistencies.

The implementation of the PIP strategy comes with non-trivial numerical choices. First, we consider the numerical issue of having constraint functions that are not well-scaled with the objective function. To address this problem, the logarithmic and penalty terms of the merit function (see~\cref{merit-z}) are multiplied by scaling constants $b^{\tt int}$ and $b^{\tt ext}$; i.e., the merit function can be written as
$$ z(\vecx;\rho) = f(\vecx) - b^{\tt int} \rho \log(-\cint(\vecx)) + \tfrac{b^{\tt ext}}{\rho} \cext(\vecx). $$
Note that such constants do not affect the theoretical properties of the algorithm.
After some preliminary tests, the choice of $b^{\tt int}$ did not seem to affect the performance of the algorithm significantly; hence, it is set equal to $1$. The remaining constant is set using the following rule,
$$b^{\tt ext} = \begin{cases}
1 & \text{if } f(\vecx_0)= 0, \\
\max\{1,10^{\lfloor\log_{10}(|f(\vecx_0|)\rfloor}\} & \text{otherwise}.

\end{cases}$$

The idea behind the latter formula is to increase the effect of the penalty term to the same order of magnitude as the value of the objective function at the initial point. This allows the objective function and the exterior constraint function to have comparable contributions to the merit function value. Note that  when the order of magnitude of $f(\vecx_0)$ is smaller than $1$, the constant $b^{\tt ext}$ is set to $1$. Setting $b^{\tt ext}$ as a power of $10$ is due to faster and more robust numerical computation.
Finally, the penalty-barrier parameter $\rho$ is initialized as $\rho_0 = 10^{-1}$.

When an iteration $k$ is unsuccessful, the algorithm performs the penalty parameter update step.
For the practical implementation, the updating criterion is relaxed by multiplying both terms by positive constants $b_{\rho}$ and $b_c$. The penalty-barrier parameter $\rho_k$ is reduced whenever
$$ \Delta_{k+1} \leq \min\{ b_{\rho}\rho_k^{\beta}, b_c[\phi^{\tt prox}(\vecx_k)]^2\},$$
where $\beta = 1+10^{-9}$, $b_\rho = 10$, and $b_c = 10^{10}$. For values of $\phi^{\tt prox}(\vecx_k)$ on the order of $10^{-5}$, the quantity $b_c[\phi^{\tt prox}(\vecx_k)]^2$ 
is on the order of $1$. This means that the influence of the second term of the criterion is not masked by the scaling constant, it is solely relaxed. Note that this choice is consistent with the theoretical analysis. Finally, the penalty contraction parameter is set as $\theta_\rho = 10^{-2}$.

We conclude this section by describing a strategy to update the current incumbent at each iteration when the unconstrained subproblem is updated. Whenever the penalty-barrier parameter is updated or the constraint index partition sets are changed, the merit function is modified, and the current incumbent point might not be the best candidate among the previously visited ones. Therefore, the algorithm computes analytically the merit function value for all the points stored during the optimization routine and selects the one with the lowest merit function value.

\section{Computational experiments}
\label{sec-tests}
This section proposes a series of computational experiments to compare
MADS with Penalty-Interior Point (MADS-PIP) to
MADS using the progressive barrier (MADS-PB).
For one of the real-life blackbox optimization problems,
the comparison considers CMA-ES without parameter tuning from the python module PYCMA\footnote{\url{https://github.com/CMA-ES/pycma}}.

We benchmark algorithms for solving constrained problems by examining their data profiles within a selected blackbox evaluation budget. 
In this work, we compute the data profile function $d_a(k)$ of~\cite[Definition~3.3]{G-2025-36} to measure the portion of problem instances that an algorithm $\tau$-solves with a precision $\tau$ within $k$ groups of $n + 1$ evaluations. The best known objective function value $f^*$ considered here is the lowest feasible value achieved by any algorithm on a problem instance. A problem instance combines a test problem with an initial point and a fixed random seed.

Since MADS-PB and MADS-PIP handle constraints in fundamentally different ways, we also examine their efficiency in obtaining a first feasible point starting from an infeasible one. In what follows, when deemed relevant, data profiles show the portion of problem instances for which a feasible point was found within $k$ groups of $n + 1$ evaluations. 

For problems with equality constraints, a strict feasibility cannot be obtained during optimization. 
During the benchmarking process, a relaxed feasibility criterion is considered.
After some preliminary testing, a point is considered feasible if $|h_j(\textbf{x})| < 10^{-8}$ for each constraint index $j$.

\subsection{Settings}
In this work, all MADS variants used the Ortho~2N polling directions~\cite{AbAuDeLe09}. Several search methods are enabled by default for the search step. In this work, the default Nelder-Mead search method~\cite{AuTr2018} is disabled to focus solely on MADS. For the comparisons, we have considered the effect of enabling the quadratic model search method~\cite{CoLed2011} or not, noted MADS-PIP QMS and MADS-PB QMS in what follows. Quadratic models are built for each output individually before any combination into a non-quadratic surrogate of the merit function $z$ (MADS-PIP) or the infeasibility measure (MADS-PB). The search step solves a surrogate subproblem to find a new candidate point within a box whose side lengths are proportional to the frame size (typically four times larger) and centred on the incumbent point $\vecx_k$. The surrogate subproblem is solved with the same algorithm as the main optimization problem~\ref{P0}. But, for MADS-PIP QMS, during the subproblem iterations the penalty parameters $\rho_k$ are fixed to values taken from the main problem iteration. 

The code of MADS progressive barrier was also adapted to handle equality constraints. The infeasibility of each equality constraint $|h_j(\textbf{x})|$ is added into the overall infeasibility measure when above a selected feasibility threshold. In other words, we have considered that if $|h_j(\textbf{x})| < \epsilon$ the constraint can be relaxed and not included in the infeasibility measure. If not, the quantity $|h_j(\textbf{x})|^2$ is added into the infeasibility measure.
Relaxing feasibility must be done with care because it alters a problem's solution. Some testing was done with several small $\epsilon$s. To be consistent with the choice made previously for benchmarking, $\epsilon=10^{-8}$ was chosen as a good compromise for all the tests on problems with equality constraints.

\subsection{Inequality constrained optimization}
\label{sec::Inequalitypbs}

This section describes the computational experiments conducted on three inequality-constrained test sets.

\subsubsection*{Results on problems from the literature}
\cref{tab-pbs} lists fourteen inequality constrained problems selected from the literature. 
The number of variables $n$ varies from $2$ to $13$. 
The number of  general inequality constraints $m$ ranges from $1$ to $15$ and there are no equality constraint.
Feasible and infeasible starting points are considered. 
A problem instance is defined by a single initial point and a random seed. Each problem with a selected initial point is run $10$ times with different seeds for a total of $780$ runs for each algorithm tested.

\begin{table}[htb!]
\small
\centering
\scriptsize
\renewcommand{\arraystretch}{1.}
\begin{tabular}{|l|c|c|c||l|c|c|c|}
\hline
Problem & $n$ & $m$ & \# $x_0$ &
Problem & $n$ & $m$ & \# $x_0$ \\
\hline\hline

CHENWANG\_F2\cite{ChWa2010}   & 8  & 6 & 11 &
MEZMONTES\cite{MezCoe05}   & 2  & 2 & 1  \\

CHENWANG\_F3\cite{ChWa2010}   & 10 & 8 & 10 &
OPTENG\_RBF\cite{KiArYa2011} & 3  & 4 & 1  \\

CRESCENT\cite{AuDe09a}    & 10 & 2  & 1  &
PENTAGON\cite{LuVl00}     & 6  & 15  & 1  \\

DISK\cite{AuDe09a}    & 10 & 1  & 1  &
SNAKE\cite{AuDe09a}    & 2  & 2  & 1  \\

G210\cite{AuDeLe07}   & 10 & 2 & 1  &
SPRING\cite{RodRenWat98}& 3  & 4 & 10 \\

MAD6\cite{LuVl00}     & 5  & 7  & 10 &
TAOWANG\_F2\cite{TaoWan08}   & 7  & 4 & 10 \\

MDO$^*$ & 10 & 10 & 10 &
ZHAOWANG\_F5\cite{ZhaWan2010b}& 13 & 9  & 10 \\
\hline
\end{tabular}
\begin{minipage}{0.8\textwidth}
$^*${\scriptsize MDO problem with the same analytical expressions for disciplines than the MDO problem in \cref{sec:mdo_eq} but with a different approach to solve the multidisciplinary analyses.
Source code available at \url{https://github.com/bbopt/aircraft_range}.}
\end{minipage}
\caption{Description of a set of analytical problems from the literature without any equality constraints.}
\label{tab-pbs}
\end{table}

\cref{fig:78_pbs} shows the obtained results on the selection of problems from the literature. 
The results show that with or without the quadratic model, 
MADS-PIP performs slightly better than MADS-PB.
On this test set, the impact of the models is much more important than the choice in the strategy (PIP versus PB) to handle the constraints.

\begin{figure}[!ht]
\centering
\includegraphics[trim={0cm 0cm 0cm 0cm 0cm},clip, width=\linewidth]{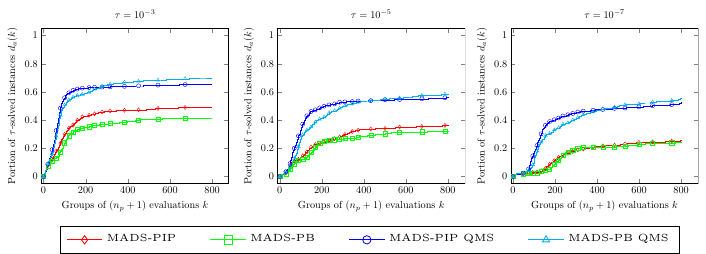}
\caption{Data profiles on a selection of analytical problems from the literature.}
\label{fig:78_pbs}
\end{figure}

\subsubsection*{Results on a real blackbox: The \solar{6} problem}

\solar{} is a blackbox optimization benchmarking framework~\cite{solar_paper}.
The \solar{6} instance is considered here for benchmarking. 
It is a simulator of a thermal power plant using the molten salt cycle and power block models. The objective is to minimize the cost of the thermal storage units so that the power plant is able to sustain a $120$MW electrical daily power output. The heliostat field is not simulated in \solar{6}. The problem has $5$ variables, $6$ general inequality constraints and no equality constraints. The thirty different initial points proposed are those of the \solar{} repository\footnote{\url{https://github.com/bbopt/solar}} and a single seed per optimization is considered. 
Hence, each algorithm is run on $30$ problem instances.

The average optimization time for each \solar{} problem instance was approximately $14$ hours on a 12th Gen Intel(R) Core(TM) i7 processor.
Each of the $1,500$ blackbox evaluations within an optimization run was executed on a single core, while multiple (at most $12$) optimization runs were performed in parallel.

\cref{fig:sola6_30X0s} includes CMA-ES to the comparison.
In this problem, introducing quadratic models seems to degrade the algorithm’s overall performance, and the induced lack of smoothness adversely affects CMA-ES.
MADS-PIP QMS outperforms MADS-PB QMS, but without models MADS-PIP is preferable to MADS-PB only when $\tau$ is small.

\begin{figure}[!ht]
\centering   
\includegraphics[trim={0.7cm 0cm 0cm 0cm 0cm},clip, width=\linewidth]{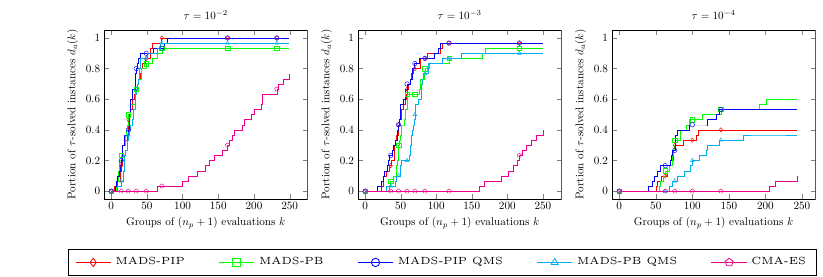}
\caption{Data profiles obtained on $30$ inequality constrained problem instances of \solar{6}.}
\label{fig:sola6_30X0s}
\end{figure}

On this first real-life blackbox problem, the difference in performance between MADS-PB and MADS-PIP is not as pronounced as in \cref{sec::Equalitypbs}.

\subsection{Equality and inequality constrained optimization}\label{sec::Equalitypbs}

This section analyzes two sets of optimization problems constrained by equalities and inequalities.

\subsubsection*{Results on CUTEst problems}

\cref{cutest:equality} lists the $25$ CUTEst problems with both equality and inequality constraints. The problem dimensions range from $4$ to $41$, the number of inequality constraints varies between $3$ and $90$, and there are up to $21$ equality constraints.
The problem instances are generated 
using five different random seeds.

\begin{table}[htp!]
\centering
\scriptsize
\renewcommand{\arraystretch}{1.}
\renewcommand{\tabcolsep}{6pt}
\begin{tabular}{|l|c|c|c||l|c|c|c||l|c|c|c|}
\hline
Problem & $n$ & $m$  & $p$
& Problem & $n$ & $m$  & $p$
& Problem & $n$ & $m$ & $p$\\
\hline \hline
ALLINQP  & 24 & 18 &  3  &
GAUSSELM & 29 & 36 &  14 &
LOADBAL  & 31 & 31 & 11 \\
ANTWERP  & 27 & 10 &   8  &
HADAMARD & 37 & 93 &  21 &
RES      & 20 & 14 &  12 \\
BLOCKQP1 & 35 & 16 & 15  &
HS74     &  4 &  5 &   3 &
SYNTHES2 & 11 & 15 &  1 \\
BLOCKQP2 & 35 & 16 &  15  &
HS75     &  4 &  5 &  3 &
SYNTHES3 & 17 & 23 &   2 \\
BLOCKQP3 & 35 & 16 & 15  &
HS114    & 10 & 11 &   3 &
TENBARS4 & 18 &  9 &    8 \\
BLOCKQP4 & 35 & 16 &  15  &
JANNSON3 & 30 &  3 &  1 &
TRUSPYR2 & 11 & 11 & 3 \\
BLOCKQP5 & 35 & 16 &  15  &
KISSING  & 37 & 78 & 12 &
ZIGZAG   & 28 & 30 &  20 \\
EG3      & 31 & 90 &   1  &
LIPPERT1 & 41 & 80 & 16 & &  & &   \\
ERRINBAR & 18 &  9 &    8  &  LIPPERT2 & 41 & 80 & 16 &  &  &  &   \\
\hline
\end{tabular}
\caption{Set of $25$ equality and inequality constrained problems selected from the CUTEst collection. The parameters  $n$, $m$, and $p$ denote, respectively, the number of variables, of inequality constraints, and of equality constraints.}
\label{cutest:equality}
\end{table}

\cref{fig:25_pbs_eq} compares MADS-PIP and MADS-PB on the $25$ CUTEst problems with equality and inequality constraints.
The plot on the left of \cref{fig:25_pbs_eq} shows the portion of instances for which the methods generate a feasible point for equality constrained problems.
Here MADS-PB is clearly outperformed by MADS-PIP.
This is not surprising because MADS-PB is not designed to handle equality constraints.
For MADS-PIP, the quadratic model helps to quickly find a feasible point.
The two plots on the right show the data profiles.  
Again, MADS-PIP is superior to MADS-PB, and the use of quadratic models is again shown to be very useful.

\begin{figure}[!ht]
\centering
\includegraphics[trim={0.7cm 0cm 0cm 0cm 0cm},clip, width=\linewidth]{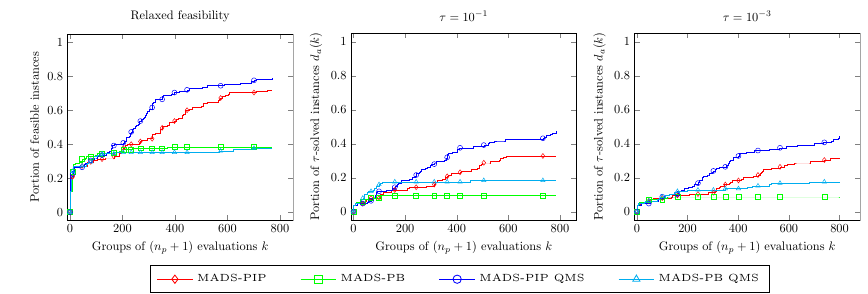}
\caption{Data profiles on $25$ CUTEst analytical problems having equality constraints. 
The left plot displays the portion of feasible instances obtained. Middle and right plots show the portion of $\tau$-solved instances for $\tau=10^{-1}$ (middle) and $\tau=10^{-3}$ (right).}
\label{fig:25_pbs_eq}
\end{figure}

\subsubsection*{Results on a real blackbox: The aircraft range MDO problem}
\label{sec:mdo_eq}
The Aircraft Range multidisciplinary design optimization (MDO) problem that has both equality and inequality constraints is also considered for benchmarking. A simulator solves the multidisciplinary analyses (MDA) to predict the operation of a supersonic business jet for a given set of variables. 
In the simulator, the aerodynamics, structure and propulsion disciplines are modelled by analytical expressions typical for an early conceptual design stage. 
The analytical expressions given in~\cite{SoSoAgSa98a} have been converted to C++ and made available in GitHub\footnote{\url{https://github.com/bbopt/aircraft_range}}.
These analytical expressions are fast to evaluate and are regarded as unknown in this blackbox simulator. 

We consider the variant of the problem in which the discipline coupling is solved by adding  interdisciplinary (coupling) variables and consistency equality constraints to the MDO problem.
This type of approach is referred to as ``individual discipline feasible'' (IDF)~\cite{Cramer1994}.
The MDO problem in \cref{sec::Inequalitypbs} uses the same discipline analyses but solves the MDA with an iterative fixed point method that is often referred to as multidisciplinary feasible (MDF).
The problem has $13$ bounded variables scaled between $0$ and $100$, $10$ general inequality constraints and $3$ general equality constraints.

Ten different initial points and ten random seeds are considered. 
A problem instance is defined by a single initial point and a random seed.
The leftmost plot in~\cref{fig:aircraft_range}
shows that again, both MADS-PB variants fail to produce feasible solutions to $70\%$ of the problem instances.
MADS-PIP has the opposite behaviour, and MADS-PIP QMS produces a feasible solution to $95\%$ of the problem instances.
The performance profiles on the right show that MADS-PB fails on 100$\%$ of the instances.  
MADS-PB QMS is clearly the dominant method.

\begin{figure}[!ht]
\centering
\includegraphics[trim={0.7cm 0cm 0cm 0cm 0cm},clip, width=\linewidth]{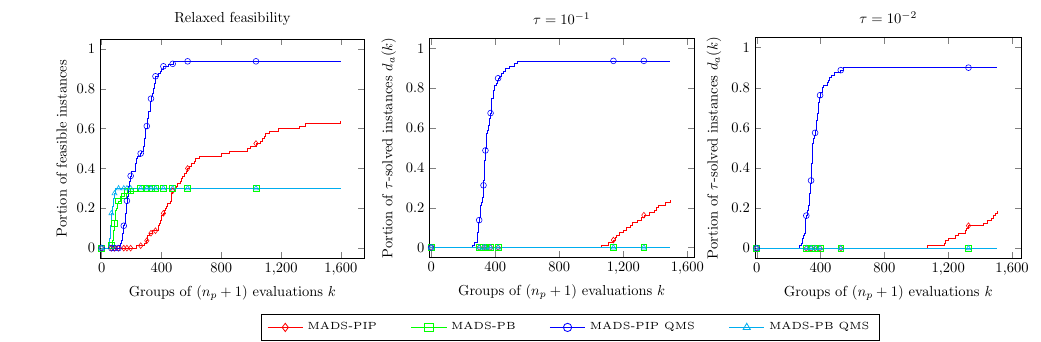}
\caption{Data profiles on 100 problem instances of Aircraft Range MDO.  
The left plot displays the portion of feasible instances obtained. Middle and right plots show the portion of $\tau$-solved instances for $\tau=10^{-1}$ (middle) and $\tau=10^{-2}$ (right).}
\label{fig:aircraft_range}
\end{figure}

In summary, the results on the real-life test problem suggest that MADS-PB and MADS-PIP are competitive on problems with only inequality constraints.  
However, as soon as equality constraints are present, MADS-PB fails as it is not designed to tackle such problems.  MADS-PIP is the first variant of MADS to handle equality constraints.

\section{Conclusions}
\label{sec-conc}
The main purpose of this work is to extend MADS to handle equality and inequality constraints efficiently.
The constraints are partitioned into two distinct sets. 
The interior set comprises certain inequality constraints, while the exterior set consists of the remaining inequality constraints together with all equality constraints. 
The interior set is managed using a logarithmic barrier based on an aggregated measure of interior violation, whereas the exterior set is treated via an exterior penalty approach.

These components are combined into a merit function, yielding the proposed penalty-interior-point method (MADS-PIP). 
This construction allows MADS to solve a sequence of unconstrained subproblems while progressively driving the penalty parameter to zero, so that the merit function approximates the original constrained problem. 

This work shows that MADS-PIP generates an infinite sequence of path-following updates and that the frame size parameters along these subsequences shrink to zero, ensuring that end-path limit points exist. 
Under mild Lipschitz continuity and density assumptions, together with appropriate constraint qualifications for both feasibility (FCQ) and stationarity (SCQ), end-path limit points are shown to be feasible for the original problem and are shown to satisfy first-order Clarke-stationarity conditions for the inequality-constrained case. 
These results provide rigorous nonsmooth convergence guaranties for the penalty interior-point strategy when combined with the MADS framework.

Furthermore, MADS-PIP was implemented in {\sf NOMAD~4}~\cite{nomad4paper}, and extensive tests were conducted on both CUTEst and other challenging blackbox problems.
These experiments indicate that MADS-PIP is competitive with the progressive barrier approach (MADS-PB) on inequality-only constrained problems and demonstrate a clear advantage when equality constraints are present.

{\small \printbibliography}


\end{document}